\documentclass[11pt, psamsfonts]{amsart}

\usepackage{amssymb, amsmath, amsthm}
\usepackage{graphicx}
\usepackage{pstricks}
\usepackage[final, hypertex]{hyperref}
\usepackage[a4paper, centering]{geometry}
\newtheorem{theorem}{Theorem}[section]
\newtheorem{prop}[theorem]{Proposition}
\newtheorem{lemma}[theorem]{Lemma}
\newtheorem{cor}[theorem]{Corollary}
\theoremstyle{definition}
\newtheorem{dhef}[theorem]{Definition}
\theoremstyle{remark}
\newtheorem{rk}[theorem]{Remark}
\newtheorem{ehse}[theorem]{Example}

\newtheorem{claim}[theorem]{{\it Claim}}
\newenvironment{claim-b}
   {\begin{claim}\rm}
   {\end{claim}}

\long\def\elimina#1{} 

\def\R{\mathbb{R}}

\def\N{\mathbb{N}}
\def\Z{\mathbb{Z}}

\def\sn{\sigma}
\def\hsn{\hat\sigma}
\def\tsn{\tilde\sigma}

\def\snell{S}
\def\ls{l}
\def\lh{\tilde l}

\def\Lspq{L}
\def\sqd{\sqrt{2}}
\def\lng{\mathcal{L}_{\beta}}

\def\floor#1{\left\lfloor #1 \right\rfloor}
\def\cray#1{{[\![#1]\!]}}
\def\ocray#1{{]\!]#1]\!]}}
\def\coray#1{{[\![#1[\![}}

\def\pscal#1#2{\langle #1,\, #2\rangle}

\newcommand{\vincolo}[3][\hsn]{\frac{#2\, #1}{\sqrt{1-{#1}^2}}+\frac{#3\,
{#1}}{\sqrt{\beta^2-{#1}^2}}}

\newcommand{\lungh}[3][\hsn]{\frac{#2}{\sqrt{1-{#1}^2}}+\frac{\beta^2\, #3
}{\sqrt{\beta^2-{#1}^2}}}

\newcommand{\rad}[2][\hsn]{\dfrac{#2}{\sqrt{1-{#1}^2}}}

\newcommand{\radb}[2][\hsn]{\dfrac{#2}{\sqrt{\beta^2-{#1}^2}}}

\newcommand{\radc}[2][\hsn]{\dfrac{#2}{(1-{#1}^2)^{3/2}}}

\newcommand{\radbc}[2][\hsn]{\dfrac{#2}{(\beta^2-{#1}^2)^{3/2}}}

\def\kc{k_c}
\def\bc{\beta^c}

\begin{document}
\title[Limits of chessboard structures]
{On the Finsler metrics obtained as limits of chessboard structures}%

\author[M.~Amar]{Micol Amar}
\address{Dipartimento di Metodi e Modelli Matematici, Univ.\ di Roma I\\
Via Scarpa, 16 -- 00161 Roma (Italy)} \email[Micol
Amar]{amar@dmmm.uniroma1.it }
\author[G.~Crasta]{Graziano Crasta}
\address{Dipartimento di Matematica ``G.\ Castelnuovo'', Univ.\ di Roma I\\
P.le A.\ Moro 2 -- 00185 Roma (Italy)} \email[Graziano
Crasta]{crasta@mat.uniroma1.it}

\author[A.~Malusa]{Annalisa Malusa}
\email[Annalisa Malusa]{malusa@mat.uniroma1.it}

\date{April 28, 2008}


\begin{abstract}
We study the geodesics in a planar chessboard structure with
two values $1$ and $\beta>1$. The results for a fixed structure allow
us to infer the properties of the Finsler metrics obtained, with an homogenization
procedure, as limit of oscillating chessboard structures.
\end{abstract}

\maketitle

\section{Introduction}

In this paper we deal with optical paths in a dioptric material
with parallel geometry and a chessboard structure on transversal
planes. Further to a bidimensional reduction, we fix the optical
features of the composite material in terms of its refractive index:
given $\beta>1$, let us define on $[0,2)\times[0,2)$ the function
\begin{equation}\label{f:defa}
{a_{\beta}}(x,y)=
\begin{cases}
\beta, & \textrm{if\ } (x,y)\in [(0,1) \times (1,2)]
\cup [(1,2) \times (0,1)], \\
1, & \textrm{otherwise},
\end{cases}
\end{equation}
and extend it by periodicity to a function defined on $\R^2$ which we
still denote by ${a_{\beta}}$.

Hence, normalizing to $1$ the speed of light in the vacuum,
light travels in the dioptric material with a speed $1/a_{\beta}$.

Since we are dealing with a system employing only refraction,
Fermat's principle dictates that the optical paths between two points minimize the
optical path length, which coincides with the time spent. Thus, in
an homogeneous material, where the speed of light is constant,
the optical path is a segment. Moreover, Fermat's principle leads
to Snell's law of refraction, which completely describes the
optical paths in layered materials (see \cite{BoWo} for a comprehensive introduction on the principles of optics).
The explicit description of the optical paths in the chessboard structure becomes harder since, for example,
no necessary condition prescribes their  behaviour at corners.

From a geometrical viewpoint, we are interested in the description of the geodesics
in the Riemannian structure $(\R^2, a_{\beta})$, that hereafter will be called
\textit{standard chessboard structure}.
We shall refer to \textit{light squares} or
to \textit{dark squares}, when, respectively, ${a_{\beta}}=1$  or ${a_{\beta}}=\beta$.

In the mathematical model, a virtual path
emanated from the origin
is a solution
$u\colon [0,T]\to\R^2$
to the differential inclusion
\begin{equation}\label{f:inclu}
\begin{cases}
u'(t) \in G_\beta(u(t)),\quad 0\leq t\leq T,\\
u(0) = 0\,,
\end{cases}
\end{equation} where the
set--valued map (with nonconvex values) $G_\beta$ is defined by
\[
G_\beta(x,y) = \frac{1}{a_\beta(x,y)}\, \partial\overline{B}_1(0),\quad
(x,y)\in\R^2
\]
(see e.g.\ \cite{AuCe} for an introduction to differential inclusions).

Fermat's Principle states that a ray of light from $O$ to $\xi$ is a solution
of the minimum time problem with target $\xi$
\begin{equation}\label{f:Tmin}
T_\beta(\xi) =
\inf\{T\geq 0;\
\exists\ u(\cdot)\ \text{solution of (\ref{f:inclu}), with}\
u(T) = \xi\}\,.
\end{equation}
In \cite{CeFM} it is proved that
the infimum in (\ref{f:Tmin}) is reached.
Moreover, if $u(\cdot)$ is a solution of (\ref{f:inclu}) satisfying
$u(T) = \xi$, then the optical length of the curve
$\Gamma = \{u(t);\ 0\leq t\leq T\}$ is
\[
\lng(\Gamma) =
\int_0^T a_{\beta}(u(t))\, |u'(t)|\, dt
= T.
\]
Then optical paths are
geodesic curves in the Riemannian structure $(\R^2, a_{\beta})$.
We underline again that the global minimization procedure considers only
refracted rays, excluding all reflected rays,
because a reflected ray is never a global geodesics. Hence our
results have a reasonable physical meaning  either if
the number of
interfaces traversed in the periodic media is not
very high, or for $\beta$ near to $1$, since in both cases the reflected light
can be neglected. Anyhow the results are intended as a depiction of curves
of minimal length in the Riemannian structure $(\R^2, a_{\beta})$.

A starting point is the elementary observation that, for $\beta > 2$, any geodesic joining two points
in the light material (i.e. in the set $\{a_{\beta}=1\}$) is
a ``light path'', i.e.\ it never crosses the dark material.
With little more work it is not difficult to prove the same conclusion if
$\beta > \sqrt{2}$, provided that the endpoints of the geodesic have integer coordinates.
Moreover, the value $\beta = \sqrt{2}$ is an optimal threshold in this class of geodesics, since the diagonal
of a dark square is a geodesic for every $\beta \leq \sqrt{2}$.



In this paper we obtain a perhaps surprising result,
focusing our attention to geodesics joining two points with
coordinates $(2n+j,j)$, $n$, $j\in\Z$, that we call
{\em light vertices}.
For $\beta \geq\sqrt{3/2}$
we depict explicitly the geodesics joining the origin
to a light vertex.
As a consequence, we prove that
the threshold value for the minimality of light paths
is given by $\bc_0\in (\sqrt{3/2}, \sqrt{2})$,
which is exactly the value of $\beta$ such that the optical
path joining $O$ with the light vertex $(3,1)$
has the same length ($2+\sqrt{2}$) of the
optimal light paths.
More precisely, we show that
for $\beta>\bc_0$ the geodesics are optimal light paths,
whereas for $\sqrt{3/2}<\beta<\bc_0$
the minimal curves are constructed
concatenating the maximal number of translations of the optical path joining $O$ with $(3,1)$,
with segments either on the sides of the squares or on light diagonals.

For $1<\beta<\sqrt{3/2}$ the characterization of the geodesics
seems to be a hard problem,
for reasons that will be clarified in Section~\ref{s:geovere}.
Anyhow, we are able to compute the optical paths
joining the origin to a light vertex in
a small cone $\{0\leq K\, y \leq x\}$,
where $K = K(\beta)$ is an odd positive integer,
whose value diverges to $+\infty$ as $\beta$ approaches $1$.

\smallskip
The knowledge of the minimal length of curves joining the origin to
light vertices
in the chessboard structure is enough to characterize the
so-called homogenized model.
Namely,
the results described above give information
about the optical paths
in an inhomogeneous dioptric material whose observed refractive index,
at a mesoscopic level, is given by the chessboard structure, that is the observed
refractive index at a given scale $\varepsilon>0$ is
$a^{\varepsilon}_{\beta}(x,y)=a_{\beta}\left(x/{\varepsilon},y/\varepsilon\right)$.
We are interested in the behavior of the optical length of geodesics
when $\varepsilon\to 0$.

We already know that, at the scale $\varepsilon$, a virtual path emanated from the origin
is a solution
$u\colon [0,T]\to\R^2$
of the differential inclusion
\begin{equation}\label{f:inclue}
\begin{cases}
u'(t) \in \dfrac{1}{a_\beta^{\varepsilon}(u(t))}\, \partial\overline{B}_1(0),
\quad 0\leq t\leq T,\\
u(0) = 0\,,
\end{cases}
\end{equation}
and Fermat's Principle states that a ray of light from $O$ to $\xi$ is a solution
of the minimum time problem with target $\xi$
\begin{equation}\label{f:Tmine}
T_\beta^{\varepsilon}(\xi) =
\inf\{T\geq 0;\
\exists\ u(\cdot)\ \text{solution of (\ref{f:inclue}), with}\
u(T) = \xi\}\,.
\end{equation}

We are interested in the characterization of the limit, as $\varepsilon \to 0$, of the minimum
time problems (\ref{f:inclue})--(\ref{f:Tmine}).

The minimum time problems can be rephrased in terms of
minimum problems of the Calculus of Variations.
Let us denote by $\mathcal{L}_\beta^\varepsilon$ the length functional in the chessboard
structure corresponding to $a_\beta^\varepsilon$,
that is
\[
\mathcal{L}_\beta^\varepsilon(u)=\int_0^1 a_\beta^\varepsilon(u(t))\, |u'(t)|\, dt\,,\qquad u\in AC([0,1],\R^2)\,,
\]
and by $d_{\beta}^{\varepsilon}(0,\xi)$ the distance between $O$ and $\xi$
in such a structure, that is
\begin{equation}\label{f:defdist}
{d_{\beta}^{\varepsilon}}(0,\xi)=\inf\left\{\mathcal{L}_\beta^\varepsilon(u):\ u\in AC([0,1];\R^2)
\text{ s.t. } u(0)=0,\ u(1)=\xi\right\}\,.
\end{equation}

If $u(\cdot)$ is a solution of (\ref{f:inclue}) satisfying
$u(T) = \xi$, then $T$ equals the optical length of the curve
$\Gamma = \{u(t);\ 0\leq t\leq T\}$
so that
\begin{equation}\label{f:Tgamma}
T_\beta^{\varepsilon}(\xi) = d_\beta^{\varepsilon}(0,\xi),\qquad
\xi\in\R^2.
\end{equation}
The advantage of this formulation is that
the asymptotic behavior of ${d_{\beta}^{\varepsilon}}$ can be discussed in terms of $\Gamma$--convergence of the functionals
$F^{\varepsilon}_{\beta}$
(see e.g.\ \cite{Brai} for an introduction to $\Gamma$--convergence).
In \cite{AmVi} it was shown that the sequence $(\mathcal{L}_{\beta}^\varepsilon)$ $\Gamma$--converges in
$AC([0,1],\R^2)$ (w.r.t. the $L^1$ topology) to the functional
\[
\mathcal{L}_\beta^{hom}(u)=\int_0^1 {\varPhi_\beta}(u'(t))\, dt\,,\qquad u\in AC([0,1],\R^2)\,,
\]
where ${\varPhi_\beta}\colon \R^2 \to [0, +\infty)$ is a convex, positively $1$-homogeneous function,
such that $|\xi|\leq {\varPhi_\beta}(\xi)\leq \beta |\xi|$ for every $\xi\in\R^2$.
As a consequence, ${\varPhi_\beta}$ turns out to be a homogeneous Finsler metric in $\R^2$
(see e.g.\ \cite{BCS} for an introduction to Finsler geometry).

In \cite{AcBu} it is shown that ${\varPhi_\beta}$ is not a
Riemannian metric in $\R^2$ for every $\beta>1$.
In this paper we shall refine this result proving that
the optical unit ball
$\{\varPhi_\beta\leq 1\}$ is neither strictly convex
nor differentiable
(see Theorem~\ref{t:edges} and Corollary~\ref{c:regul} below).

Since ${\varPhi_\beta}$ is characterized by
\[
{\varPhi_\beta}(\xi)=\lim_{\varepsilon \to 0^+} {d_{\beta}^{\varepsilon}}(0,\xi)\,,
\]
then the limit of the minimum
time problems (\ref{f:inclue})--(\ref{f:Tmine}) is given by
\[
T_\beta(\xi) =
\inf\{T\geq 0;\
\exists\ u(\cdot)\ \text{such that\ } \varPhi_\beta(u'(t))=1,\ u(0)=0,\ u(T)=\xi\}
={\varPhi_\beta}(\xi)\,.
\]
This is what we have called homogenized model related to the chessboard structure.

Since ${\varPhi_\beta}$ is positively 1-homogeneous, it is completely determined by the geometry of
the optical unit sphere $\{{\varPhi_\beta} = 1\}$, which is, in some sense, a generalized geometric
wavefront with source point located at $O$.
Moreover, due to elementary symmetry properties,
it is enough to describe the set
$\{{\varPhi_\beta} = 1\}\cap \{0\leq y\leq x\}$.

Starting from our results on chessboard structures,
we obtain that,
if $\beta\geq \bc_0$, then the homogenized metric is
\begin{equation}\label{fottag}
{\varPhi_\beta}(x,y)=(\sqd-1)\min\{|x|,|y|\}+\max\{|x|,|y|\}\,,
\qquad \forall\ (x,y)\in\R^2\,,
\end{equation}
and the geometric wavefront $\{{\varPhi_\beta}=1\}$ is the regular octagon
inscribed in the unit circle, whose vertices lie on the coordinate axes and on the diagonals.
This result generalizes
the trivial remark concerning the case $\beta\geq 2$ (see \cite{Brai}).
On the other,
we obtain that the octagonal geometry of the wavefront
breaks for $\beta=\bc_0$,
and the optical unit sphere becomes an irregular polygon with
sixteen sides for $\sqrt{3/2} \leq \beta < \bc_0$
(see Figure~\ref{fig:ottag}).
These are two of the main results of our paper;
we refer to Theorem~\ref{t:Phi} below for their
precise statement.

For $1<\beta<\sqrt{3/2}$
we will be able to compute the optical unit ball
in the two small cones $\{K\, |y| \leq |x|\}$
and $\{K\, |x| \leq |y|\}$,
where $K=K(\beta)$ is the odd positive integer
introduced above for the chessboard structure.
More precisely, we shall show that
in these regions
the boundary $\{{\varPhi_\beta}=1\}$ is piecewise flat
with corners at the points
$(1,0)$, $(0,1)$, $(-1,0)$, and $(0,-1)$
(see Theorem~\ref{t:edges}).
As a consequence,
the optical unit ball
$\{\varPhi_\beta\leq 1\}$ is neither strictly convex
nor differentiable.

We conclude this introduction with a warning
on the physical interpretation of our results concerning the homogenized
model.
First of all, the geometrical optics approximation
is valid if the lengthscale $\varepsilon$ is much
greater than the wavelength of light.
If $\varepsilon$ is of the same order of magnitude
of the wavelength, then we fall in the domain
of photonic crystals optics, for which the full
system of Maxwell equations must be considered.
On the other hand, even in the range of
geometric optics,
our global minimization procedure considers only
refracted rays, excluding all reflected rays.
Since, at a macroscopic level, the number of
interfaces traversed in the periodic media can be
very high, the reflected light
cannot in general be neglected.


\bigskip
In the paper the following notation will be used.
\begin{itemize}
\item
$\cray{P,Q}$: closed segment joining $P, Q\in\R^2$

\item
$\ocray{P,Q} = \cray{P,Q}\setminus\{P\}$, $\coray{P,Q}= \cray{P,Q}\setminus\{Q\}$

\item
$\cray{P_1, P_2, \ldots, P_n}$:
polygonal line joining the ordered set of points
$P_1, P_2,\ldots,P_n$.

\item
$\floor{t}=\max \{k\in \Z\ \colon\ k\ \leq t\}$.

\item
$f_s=\dfrac{\partial f}{\partial s}$: partial derivative of a function $f$
with respect to $s$

\item
$\snell(A,B)$: Snell path joining $A$ to $B$
(defined in Section~\ref{s:lay})

\item
$\lng(\Gamma)=\lng^1(\Gamma)$: length of $\Gamma$ in the standard chessboard structure (optical length).
\end{itemize}

\section{Snell paths}\label{s:lay}


Let us consider the flat Riemannian structure $(\R^2,\overline{a}_{\beta})$, where
\[
\overline{a}_{\beta}(x,y)=
\begin{cases}
1 & \textrm{if \ }\floor{x} \textrm{\ is even},\\
\beta& \textrm{if \ } \floor{x} \textrm{\ is odd}\,,
\end{cases}
\qquad \beta\in\R,\ \beta>1\,.
\]

Note that this structure corresponds to a composite medium whose structure is made
by alternate vertical strips of light and dark material.

It is well known that for every pair $A=(x_A,y_A)$, $B=(x_B,y_B)$ there exists a
unique curve
of minimal length (the geodesic curve) joining $A$ to $B$,
which is an affine path in every vertical strip $\{\floor{x}=k\}$, $k\in\Z$.
Moreover, at every
interface between two strips, the change of slope
is governed by the Snell's Law of refraction
\begin{equation}\label{f:snell}
\sin \theta_1 =\beta \sin \theta_2\,,
\end{equation}
where $\theta_1$ and $\theta_2$ are the angles of incidence with the interface
from the  light strip and
the  dark strip, respectively (see, e.g., \cite[\S3.2.2]{BoWo} or \cite[\S3.4]{Ces}).

In the sequel, this geodesic curve
will be called the \textit{Snell path} joining $A$ and $B$
and will be denoted by  $\snell(A,B)$.
In order to fix the ideas, we shall always assume that $x_A<x_B$, and $y_A\leq y_B$.
We shall refer to the positive quantity
$x_B-x_A$ as the \textit{thickness} of the Snell path.

\begin{figure}
\includegraphics[height=4cm]{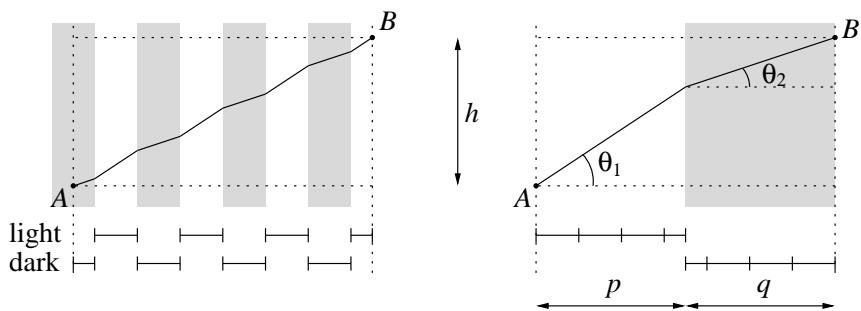}
\caption{A Snell path in the layered material and the equivalent Snell path with
a single interface.}
\label{figSnellstrips}
\end{figure}

Let $p$, $q\geq 0$ be, respectively, the thickness of the light and of the dark zone
crossed by the path $\snell(A,B)$, so that $x_B-x_A=p+q$, and
let $h$ be the vertical height $h=y_B-y_A$.
Since $h=p\tan \theta_1+q \tan \theta_2$ (see Figure \ref{figSnellstrips}), and
(\ref{f:snell}) holds true,
then $\hsn(p,q,h)=\sin \theta_1$
is implicitly determined in term of $p$, $q$ and $h$
by the constraint
\begin{equation}\label{f:vincolo}
\vincolo{p}{q}=h\,.
\end{equation}
Clearly, $\hsn(p,q,h)$ is a strictly decreasing function w.r.t. $p$ and $q$, while it
is a strictly increasing function w.r.t. $h$.

Finally, the optical length of the Snell path $\snell(A,B)$,
given by $p/\cos\theta_1+\beta q/\cos\theta_2$,
can be expressed in terms of $p$, $q$ and $h$
taking again into account \eqref{f:snell}:
\begin{equation}\label{f:lung}
\Lspq(p,q,h) := \lungh[\hsn(p,q,h)]{p}{q}\,.
\end{equation}

%

\begin{lemma}\label{l:lpq}
Let $\hsn$ and  $\Lspq$ be defined by (\ref{f:vincolo}) and (\ref{f:lung}) respectively.
Then
\[
{\Lspq}_p(p,q,h)=\sqrt{1-\hsn^2}\,, \qquad {\Lspq}_q(p,q,h)=\sqrt{\beta^2-\hsn^2}\,,
\]
for every $p$, $q\geq 0$, and for every $h\in\R$.
\end{lemma}

\begin{proof}
Differentiating (\ref{f:vincolo}) w.r.t. $p$, we get
\[
\frac{\hsn}{\sqrt{1-\hsn^2}}+\frac{p\,\hsn_p}{(1-\hsn^2)^{3/2}}
+\frac{\beta^2\, q\, \hsn_p}{(\beta^2-\hsn^2)^{3/2}}=0\,.
\]
Hence
\begin{equation*}
\begin{split}
\Lspq_p(p,q,h) & =\frac{1}{\sqrt{1-\hsn^2}}+\frac{p\,\hsn_p\hsn}{(1-\hsn^2)^{3/2}}
+\frac{\beta^2\, q\, \hsn_p\hsn}{(\beta^2-\hsn^2)^{3/2}} \\
& = \frac{1}{\sqrt{1-\hsn^2}}- \frac{\hsn^2}{\sqrt{1-\hsn^2}}= \sqrt{1-\hsn^2}\,.
\end{split}
\end{equation*}
By an analogous computation one obtains the expression for $\Lspq_q$.
\end{proof}

\begin{rk}\label{r:thfix}
By Lemma \ref{l:lpq}, it follows that, given the thickness $\tau=p+q$
and the height $h\in\R$ of a Snell path
we have that
\[
\frac{d}{dq} \Lspq(\tau-q, q,h) = -\sqrt{1-\hsn^2}+\sqrt{\beta^2-\hsn^2}>0\,.
\]
The geometrical meaning of this formula is clear:
for  Snell paths with fixed thickness,
the more is the thickness of the dark material crossed,
the more is the optical length of the path.
\end{rk}

\section{The normalized length}

Let us consider now $\R^2$ endowed with the standard chessboard structure.

\begin{dhef}
The $n$-th \textit{light diagonal}, $n\in\Z$, is the straight line $D_n$ of equation $y=x-2n$.
A \textit{light vertex} is a point having integer coordinates and belonging to a light
diagonal.
\end{dhef}

\begin{dhef}
Given two points $A$ and $B$ in the same horizontal strip $\{y\in [r,r+1]\}$, $r\in\Z$,
the Snell path joining $A$ to $B$ (in the chessboard structure) is the geodesic $\snell(A,B)$ in the corresponding
parallel layer structure 
$\overline{a}(x+r,y)$.
\end{dhef}

We are interested in the properties of the Snell paths starting from a light vertex $A$ (say $A=O$, without loss of generality)
and ending in a point $B$ on the other side of the horizontal strip containing $A$ (say $B=(x_B,1)$, $x_B > 0$)
(see Figure \ref{fig:figLen}).

\begin{figure}
\includegraphics[height=3cm]{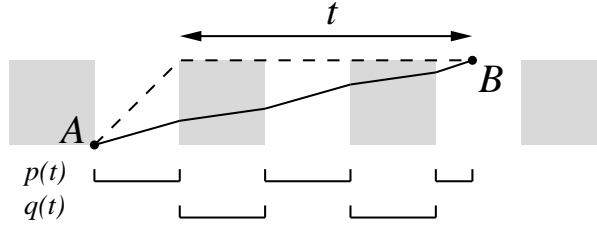}
\caption{The solid line corresponds to the Snell path $\snell(A,B)$,
while the dashed line is a minimal path joining $A$ and $B$
without crossing the dark squares.}
\label{fig:figLen}
\end{figure}

If $0<x_B\leq 1$, clearly $\snell(O,B)=\cray{O,B}$ is the unique geodesic joining $O$ to $B$.
On the other hand, $\snell(O,B)$ need not to be a geodesic when $x_B>1$.
Namely, we already know
that for $\beta$ large enough the optical length of $\snell(O,B)$ is strictly greater than
the optical length of the path obtained
by a concatenation of horizontal segments on the lines $x=0$ or $x=1$, with total length $x_B-1$,
and a segment on a light diagonal, with length $\sqrt{2}$.
In this section we discuss the behavior of the difference
$\lng(\snell(O,B))-x_B+1-\sqrt{2}$ for $x_B\geq 1$.
To this aim, for $t\geq 0$ and $\beta\geq 1$,
with some abuse of notation we define
$\hsn(t,\beta)= \hsn(p(t), q(t),1)$,
where
\begin{equation}\label{f:defq}
q(t)=
\begin{cases}
\frac{\floor{t+1}}{2} & \textrm{if\ } \floor{t}\ \textrm{is odd}, \\
t-\frac{\floor{t}}{2} & \textrm{if\ } \floor{t}\ \textrm{is even},
\end{cases}
\end{equation}
and $p(t)=t+1-q(t)$.
Recall that $\hsn(t,\beta)$ is determined by the constraint
\begin{equation}\label{f:vinct}
\vincolo{p(t)}{q(t)}-1=0\,.
\end{equation}
As a consequence we have
\begin{equation}\label{f:deriv}
\hsn_t(t,\beta) = \frac{{-\dfrac{\hsn p_t}{\sqrt{1-\hsn^2}}-\dfrac{\hsn q_t}{\sqrt{\beta^2-\hsn^2}}}}
{{\dfrac{p}{(1-\hsn^2)^{3/2}}+\dfrac{\beta^2 q}{(\beta^2-\hsn^2)^{3/2}}}} <0\,, \qquad t>0,\ t\not\in\N\,,
\end{equation}
since $p_t(t)$, $q_t(t)\geq 0$, and $p_t(t)+q_t(t)=1$ for every $t>0$, $t\not\in\N$ .
Thus the map $t\mapsto \hsn(t,\beta)$ is strictly decreasing in $[0,+\infty)$,
and satisfies $\hsn(0,\beta) = 1/\sqd$,
$\lim_{t\to +\infty} \hsn(t,\beta) = 0$ (see Figure \ref{fig:sigma}).

Moreover, as a straightforward consequence of the Implicit Function Theorem,
we have that $\hsn_{\beta}(t,\beta)>0$ for every $t>0$. In particular we have
\begin{equation}\label{f:sigbeta}
\hsn(t,\beta)>\hsn(t,1)=\frac{1}{\sqrt{(1+t)^2+1}}\,,\qquad \forall t>0,\
\forall \beta>1.
\end{equation}

\begin{figure}
\includegraphics[height=4cm]{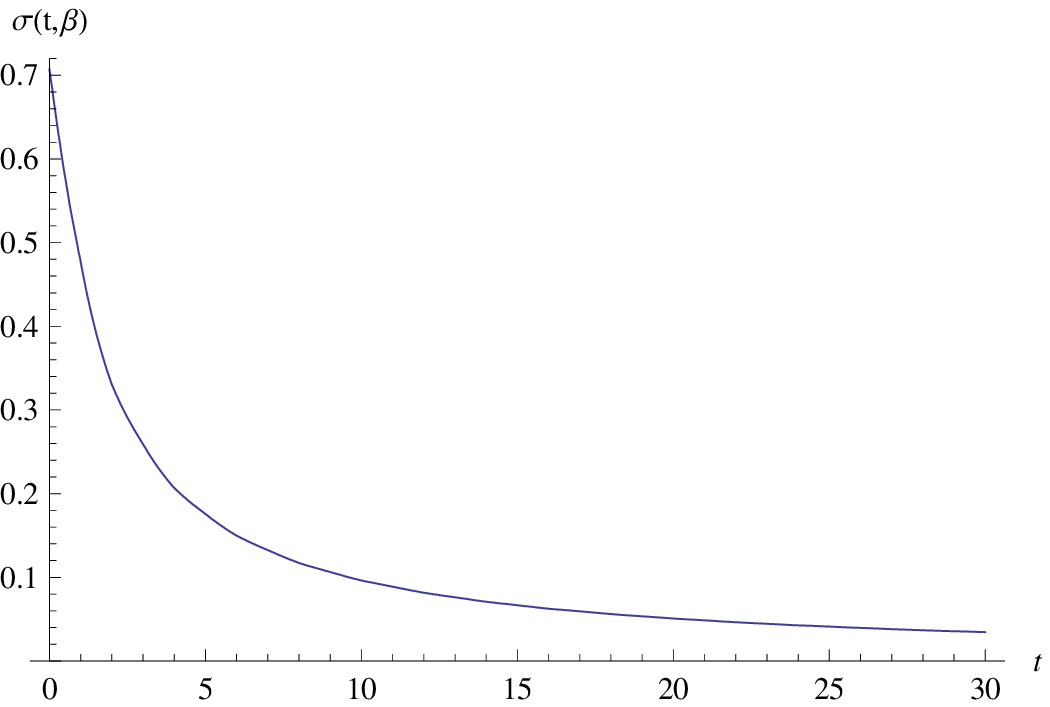}
\caption{Plot of $\hsn(t,\beta)$, $\beta=1.15$.}
\label{fig:sigma}
\end{figure}

\begin{dhef}
We shall call \textit{normalized length} the function
\begin{equation}\label{f:defelle}
\ls(t,\beta):= 
\lungh[\hsn(t,\beta)]{p(t)}{q(t)}-t-\sqd\,.
\end{equation}
\end{dhef}

Notice that $l(t,\beta)$ is the length of the Snell path joining the origin $(0,0)$ with
the point $(t+1,1)$
normalized by subtracting the minimal length of the paths joining the same two points
without crossing the dark squares (see Figure \ref{fig:figLen}).

In order to simplify the notation we introduce the sets
\[
I_L=\bigcup_{k\in\N}(2k+1,2k+2), \qquad
I_D=\bigcup_{k\in\N}(2k,2k+1)\,.
\]
If $t\in I_L$
then the last segment
of the Snell path is in the interior of
a light square, while, if $t\in I_D$, it is in the interior of
a dark square.

The basic properties of the normalized length are collected in the following lemma.

\begin{lemma}\label{r:spiegaz1}
The following properties hold.
\begin{itemize}
\item[(i)] $\ls_t(t,\beta)=\sqrt{1-\hsn^2(t,\beta)}-1$ for every $t\in I_L$;
\item[(ii)] $\ls_t(t,\beta)=\sqrt{\beta^2-\hsn^2(t,\beta)}-1$ for every $t\in I_D$;
\item[(iii)] $\ls(\cdot,\beta)$ is strictly convex in any interval of $I_L\cup I_D$;
\item[(iv)] $\ls(\cdot,\beta)$ is strictly monotone decreasing in any interval of $I_L$;
\item[(v)] if $\beta \geq \sqrt{3/2}$, then $\ls(\cdot,\beta)$ is strictly monotone increasing
in any interval of $I_D$;
\item[(vi)] if $1<\beta < \sqrt{3/2}$, then there exists a unique $t_0>0$,
characterized by $\hsn(t_0,\beta)=\sqrt{\beta^2-1}$, and such that
\[
\begin{split}
\ls_t(t,\beta)<0, & \quad \forall\ t\in [0,t_0)\cap I_D\,, \\
\ls_t(t,\beta)>0, & \quad \forall\ t\in (t_0,+\infty)\cap I_D\,.
\end{split}
\]
\end{itemize}
\end{lemma}

\begin{proof}
The derivatives in (i) and (ii) follow from Lemma \ref{l:lpq}, upon observing that
$p_t=1$ and $q_t=0$ in $I_L$, while $p_t=0$ and $q_t=1$ in $I_D$. Clearly (i) implies (iv),
while (ii) and the fact that $0\leq \hsn^2(t,\beta)\leq {1}/{2}$ imply (v) and (vi).

(iii) follows from (i), (ii), and the fact that $\hsn(\cdot,\beta)$ is a decreasing function.
\end{proof}

In conclusion, since $\ls(0,\beta)=0$ for every $\beta\geq 1$, by Lemma \ref{r:spiegaz1}  we have that,
for $\beta\geq \sqrt{3/2}$,
$\ls(t,\beta)>0$ for $t\in (0,1)$ and
the local minima of $\ls(\cdot,\beta)$ are attained at $t=2k$, $k\in \N$, corresponding to the Snell
paths ending in the light vertices (see Figures \ref{figp2} and \ref{figp3}).

On the contrary, if $\beta <\sqrt{{3}/{2}}$, a new local minimum for $\ls(\cdot,\beta)$ may appear
(see Figures~\ref{figp1} and~\ref{figp4}).
One may wonder if $\ls(t_0,\beta)$ is an absolute minimum for some $\beta$.
The following result shows that this is never the case.
(We warn the reader that the proof is rather long and technical,
and can be skipped in a first reading.)

\begin{theorem}\label{t:t1}
Given $1<\beta<\sqrt{3/2}$, let $t_0>0$ be as in Lemma \ref{r:spiegaz1}(vi).
Then
\[
\ls(2k_0+2,\beta)\leq\ls(t_0,\beta),
\]
where $k_0=\min \{k\in\N\colon\ t_0\leq 2k+2\}$.
Moreover, the strict inequality holds if $t_0\neq 2k_0+2$.
\end{theorem}

\begin{proof}
If $t_0\in[2k_0+1,2k_0+2]$, then by Lemma \ref{r:spiegaz1}(iv),(vi),
$\ls(2k_0+2,\beta) < \ls(t,\beta)$ for every $t\in [0,2k_0+2)$,
and the result is straightforward.

Let us now consider the case $t_0\in (0,1)$, so that $k_0=0$.

Recalling that $\hsn(t_0,\beta)=\sqrt{\beta^2-1}$, we obtain that
$\ls(t_0,\beta)= \sqrt{2-\beta^2}+\sqrt{\beta^2-1}-\sqrt{2}$. On the other
hand we have
\[
\begin{split}
\ls(2,\beta)&= \frac{2}{\sqrt{1-\sn_3^2}}+\frac{\beta^2}{\sqrt{\beta^2-\sn_3^2}}-2-\sqrt{2}
-\sn_3\left(\frac{2\sn_3}{\sqrt{1-\sn_3^2}}+\frac{\sn_3}{\sqrt{\beta^2-\sn_3^2}}-1\right) \\
& = 2\sqrt{1-\sn_3^2}+\sqrt{\beta^2-\sn_3^2}-2-\sqrt{2}-\sn_3\,,
\end{split}
\]
where $\sn_3=\hsn(2,\beta)$, and we have used the constraint (\ref{f:vinct}) satisfied by $\sn_3$.
Hence, denoting by
\begin{equation}
\varphi(\sn, \beta)
= 2\sqrt{1-\sn^2}+\sqrt{\beta^2-\sn^2}-2-\sn-\sqrt{2-\beta^2}-\sqrt{\beta^2-1}\,,
\end{equation}
we have to show that $\ls(2,\beta)-\ls(t_0,\beta)=\varphi(\sn_3, \beta)<0$.

One can easily check that $\varphi(\sn, \beta)$ is strictly monotone decreasing w.r.t.\ $\sn$ in $[0,1]$, so that,
by (\ref{f:sigbeta}),
we obtain $\varphi(\sn_3,\beta)<\varphi(1/\sqrt{10},\beta)$ for every $\beta \geq 1$.
In addition, $\varphi(1/\sqrt{10},\beta)$
is a strictly monotone increasing function w.r.t. $\beta$, so that we get
\[
\varphi(\sn_3, \beta)<\varphi(1/\sqrt{10},\beta)<\varphi(1/\sqrt{10},\sqrt{3/2})=\frac{5}{\sqrt{10}}-2+
\sqrt{\frac{14}{10}}-\sqrt{2}<0\,,
\]
for every $\beta \in (1, \sqrt{3/2})$, which concludes the proof for $t_0\in(0,1)$.

Assume now that $t_0\in (2k_0, 2k_0+1)$ with $k_0\geq 1$.

Since we have
\[
\begin{split}
\ls(2k_0+2, \beta) &= \ls(t_0,\beta)+\int_{t_0}^{2k_0+1}\ls_t(t,\beta)\, dt+\int_{2k_0+1}^{2k_0+2}\ls_t(t,\beta)\, dt \\
&= \int_{t_0}^{2k_0+1}\left(\sqrt{\beta^2-\hsn^2(t,\beta)}-1\right)\, dt +
\int_{2k_0+1}^{2k_0+2}\left(\sqrt{1-\hsn^2(t,\beta)}-1\right)\, dt\,,
\end{split}
\]
our aim is to prove that
\begin{equation}\label{f:mainstlo}
\int_{t_0}^{2k_0+1}\sqrt{\beta^2-\hsn^2(t,\beta)}\, dt + \int_{2k_0+1}^{2k_0+2}\sqrt{1-\hsn^2(t,\beta)}\, dt
< 2k_0+2-t_0\,.
\end{equation}
We split the proof of (\ref{f:mainstlo}) into three steps.

\smallskip
\noindent\textsl{Step 1.} Setting
\begin{equation}\label{f:logi}
\begin{split}
f & (k,t,\gamma) = 2k+1+(\beta^2-1) t+\beta^2  \\
& -\beta\sqrt{(2k+2)^2+(t+1)^2(\beta^2-1)}
-(k+1)\log\left(\frac{1+\sqrt{1+(\gamma-1)e^{-2/(k+1)}}}{\sqrt{\gamma}+1}\right)
\end{split}
\end{equation}
and $c={1}/(1-\hsn(2k_0+1,\beta)^2)$,
we show that if $f(k_0,t_0,c)\geq 0$ then (\ref{f:mainstlo}) holds.

\noindent\textsl{Step 2.} Setting
\begin{equation}\label{f:forg}
g(b)=1+3b^2-4b\sqrt{c_0}+(1-e^{-1})\,\frac{\sqrt{c_0}-1}{\sqrt{c_0}}\,,
\qquad c_0=c_0(b)=1+\frac{9}{16}(b^2-1),
\end{equation}
then $f(k_0,t_0,c)\geq g(\beta)$.

\noindent\textsl{Step 3.} $g(b)>0$ for every $b\in(1,\sqrt{3/2})$.

\smallskip
\noindent\textsl{Proof of Step 1.}
Let us consider the functions
\[
\psi(t)=\sqrt{\beta^2-\hsn^2(t,\beta)}\,,\qquad
\chi(t)=\sqrt{1-\hsn^2(t,\beta)}\,.
\]
Recalling (\ref{f:deriv}),
and taking into account that $p_t(t)=0$ and $q_t(t)=1$ for $t\in [t_0,2k_0+1)$, we obtain
\[
\psi'= -\frac{\hsn_t \hsn}{\sqrt{\beta^2-\hsn^2}}=\frac{\hsn^2}{\beta^2-\hsn^2}\cdot \frac{1}{\dfrac{p}{(1-\hsn^2)^{3/3}}
+\dfrac{\beta^2 q}{(\beta^2-\hsn^2)^{3/3}}}
<\frac{\hsn^2 \sqrt{\beta^2-\hsn^2}}{\beta^2(t+1)},
\]
where in the last inequality we have used the fact that
the function $b \mapsto b^2/(b^2-\hsn^2)^{3/2}$ is
strictly monotone decreasing,
and $p+q= t+1$. In conclusion we obtain that $\psi$ satisfies the differential inequality
\begin{equation}\label{f:odeb}
\begin{cases}
\psi'<\dfrac{1}{t+1}\psi-\dfrac{1}{\beta^2(t+1)}\,\psi^3, & t\in [t_0,2k_0+1),\\
\psi(t_0)=1.
\end{cases}
\end{equation}

Similarly, recalling that $p_t(t)=1$ and $q_t(t)=0$ for $t\in (2k_0+1,2k_0+2)$, we obtain
\[
\chi'= -\frac{\hsn_t \hsn}{\sqrt{1-\hsn^2}}=\frac{\hsn^2}{1-\hsn^2}\cdot \frac{1}{\dfrac{p}{(1-\hsn^2)^{3/3}}
+\dfrac{\beta^2 q}{(\beta^2-\hsn^2)^{3/3}}}
<\frac{\hsn^2}{p} \sqrt{1-\hsn^2},
\]
and, since $p\geq k_0+1$, we conclude that $\chi$ satisfies the differential inequality
\begin{equation}\label{f:odew}
\begin{cases}
\chi'<\dfrac{\chi-\chi^3}{k_0+1}\,, & t \in (2k_0+1,2k_0+2),\\
\chi(2k_0+1)=\sqrt{1-\beta^2+\psi^2(2k_0+1)}.
\end{cases}
\end{equation}
Solving the Cauchy problems associated to the differential inequalities (\ref{f:odeb}),
(\ref{f:odew}), we get
\begin{gather}
\psi(t)\leq\frac{1}{\sqrt{\dfrac{1}{\beta^2}+\left(1-\dfrac{1}{\beta^2}\right)\dfrac{(t_0+1)^2}{(t+1)^2}}},
\qquad t\in[t_0,2k_0+1],\label{f:dispsi}\\
\chi(t)\leq \frac{1}{\sqrt{1+\left(\dfrac{1}{1-\beta^2+\psi_1^2}-1\right)e^{-2(t-2k_0-1)/(k_0+1)}}},
\qquad t\in [2k_0+1,2k_0+2].
\end{gather}
As a consequence of these estimates we obtain
\begin{gather*}
\int_{t_0}^{2k_0+1}\sqrt{\beta^2-\sn^2(t,\beta)}\, dt
\leq \beta\left(\sqrt{(2k_0+2)^2+(t_0+1)^2(\beta^2-1)}-\beta(t_0+1)\right)\\
\int_{2k_0+1}^{2k_0+2}\sqrt{1-\sn^2(t,\beta)}\, dt \leq 1+
(k_0+1)\log\left(\frac{1+\sqrt{1+(c-1)e^{-2/(k_0+1)}}}{\sqrt{c}+1}\right)\,,
\end{gather*}
where $c={1}/(1-\beta^2+\psi(2k_0+1)^2)$, which concludes the proof of the Step 1.

\smallskip
\noindent\textsl{Proof of Step 2.}
{}From (\ref{f:dispsi}) and the fact that
$t_0>2k_0$, we have that
\[
\psi(2k_0+1) \leq \frac{\beta(2k_0+2)}{\sqrt{(2k_0+2)^2+(\beta^2-1)(t_0+1)^2}}
< \frac{\beta(2k_0+2)}{\sqrt{(2k_0+2)^2+(\beta^2-1)(2 k_0+1)^2}},\\
\]
so that
\[
\begin{split}
c= & \frac{1}{1-\beta^2+\psi(2k_0+1)^2}
\geq \frac{1+(\beta^2-1)\left(\dfrac{2k_0+1}{2k_0+2}\right)^2}{1-(\beta^2-1)^2\left(\dfrac{2k_0+1}{2k_0+2}\right)^2}
\\ & > 1+(\beta^2-1) \left(\frac{2k_0+1}{2k_0+2}\right)^2\geq 1+\frac{9}{16}(\beta^2-1)
= c_0(\beta)\,.
\end{split}
\]

Moreover, it can be easily checked that the function
\[
\gamma \mapsto\log\left(\frac{1+\sqrt{1+(\gamma-1)e^{-2/(k_0+1)}}}{\sqrt{\gamma}+1}\right)\,,\qquad \gamma>1
\]
is strictly monotone decreasing, while
the function
\[
t \mapsto (\beta^2-1) t -\beta\sqrt{(2k_0+2)^2+(t+1)^2(\beta^2-1)},\qquad t\in (2k_0, 2k_0+1)
\]
is strictly monotone increasing. Hence we have that
\[
\begin{split}
f(k_0,t_0,c) > {} &
f(k_0,2k_0,c_0(\beta)) = 2k_0\beta^2 -\beta\sqrt{(2k_0+2)^2+(2k_0+1)^2(\beta^2-1)} \\
& +1+\beta^2
- (k_0+1)\log\left(\frac{1+\sqrt{1+(c_0-1)e^{-2/(k_0+1)}}}{\sqrt{c_0}+1}\right)\,,
\end{split}
\]
where $c_0=c_0(\beta)$ is defined as in (\ref{f:forg}).

In addition, the functions
\[
\begin{split}
k  & \mapsto 2k\beta^2 -\beta\sqrt{(2k+2)^2+(2k+1)^2(\beta^2-1)}\,,\\
k & \mapsto (k+1)\log\left(\frac{1+\sqrt{1+(c_0-1)e^{-2/(k+1)}}}{\sqrt{c_0}+1}\right)\,,
\end{split}
\]
are strictly monotone increasing for $k\geq 1$.
Hence we get
\[
\begin{split}
f(k_0,t_0,c)>
1+3\beta^2-\beta\sqrt{16+9(\beta^2-1)}-
2\log\left(\frac{1+\sqrt{1+(c_0-1)e^{-1}}}{\sqrt{c_0}+1}\right)\,.
\end{split}
\]
Finally
\[
\begin{split}
& \log\left(\frac{1+\sqrt{1+(c_0-1)e^{-1}}}{\sqrt{c_0}+1}\right)\leq
\frac{\sqrt{1+(c_0-1)e^{-1}}-\sqrt{c_0}}{\sqrt{c_0}+1} \\
& = \sqrt{c_0}\, \frac{\sqrt{1+(1-e^{-1})\frac{1-c_0}{c_0}}-1}{\sqrt{c_0}+1}
\leq \frac{1}{2} (1-e^{-1})\,\frac{1-\sqrt{c_0}}{\sqrt{c_0}}\,,
\end{split}
\]
so that the proof of Step~2 is complete.

\smallskip
\noindent\textsl{Proof of Step 3.}
We have that
\[
g'(b)=6b-\frac{7+18 b^2}{\sqrt{7+9 b^2}}+\frac{36(1-e^{-1})b}{(7+9 b^2)^{3/2}}\,,
\]
and, for $1<b\leq\sqrt{3/2}$,
\[
\begin{split}
\frac{d}{db}\left(6b-\frac{7+18 b^2}{\sqrt{7+9 b^2}}\right) & =
6-27 b \frac{7+6b^2}{(7+9b^2)^{3/2}}
= 6 - 9\, \frac{3b}{\sqrt{7+9b^2}}\, \frac{7+6b^2}{7+9b^2}
\\ & \geq
6-9\, \frac{3\sqrt{3}}{\sqrt{41}}\, \frac{13}{16}>0\,,
\end{split}
\]
and
\[
\frac{d}{db}\left(\frac{b}{(7+9 b^2)^{3/2}}\right)=-16 \frac{18b^2-7}{(7+9b^2)^{3/2}}<0, \qquad b>1\,.
\]
Then we get
\[
g'(b)>6-\frac{25}{4}+\frac{72 \sqrt{3}}{41^{3/2}}(1-e^{-1})> 0,\qquad
\forall 1<b\leq \sqrt{3/2}\,.
\]
Hence $g(b) > g(1) = 0$ for all $1<b\leq \sqrt{3/2}$, and Step 3 is proved.
\end{proof}

Now we focus our attention to the study of the sequence $\ls(2k,\beta)$, $k\in \N$.
Given $k\in\N$, we set
$\delta(k,\beta)=\ls(2k+2,\beta)-\ls(2k,\beta)$, that is
\begin{equation}\label{f:delta}
\delta(k,\beta)=\lungh[\tau]{k+2}{(k+1)}-\rad[\sn]{k+1}-\radb[\sn]{\beta^2\, k}-2\,,
\end{equation}
where $\tau = \tau(k,\beta):= \hsn(2k+2,\beta)$ and $\sn = \sn(k,\beta):=\hsn(2k,\beta)$ are implicitly defined by
\begin{gather}
\vincolo[\tau]{(k+2)}{(k+1)}=1\,, \label{f:vinca} \\
\vincolo[\sn]{(k+1)}{k}=1\,. \label{f:vincb}
\end{gather}
By the monotonicity of the function $\hsn(\cdot,\beta)$ (see inequality \eqref{f:deriv})
it follows that $\tau<\sn$.

\begin{rk}\label{r:r2}
While the sign of $\ls(t,\beta)$ gives a comparison between the optical lengths
of the Snell
path $\snell(O,(t+1,1))$ and the ``light path''
$\cray{O,(1,1),(1+t,1)}$,
the sign of $\delta(k,\beta)$ gives a comparison between
the optical lengths of $\snell(O,(2k+3,1))$
and of $\snell(O,(2k+1,1))\cup \cray{(2k+1,1),(2k+3,1)}$.
\end{rk}

\begin{rk}\label{r:rdd}
Given $\beta>1$, consider the function $\tilde{\ls}\colon [0,+\infty)\to\R$,
affine on each interval $[2k,2k+2]$ and such that $\tilde{\ls}(2k) = \ls(2k,\beta)$,
$k\in\N$.
Then the derivative of $\tilde{\ls}(t)$, for $t\in (2k, 2k+2)$,
is given by $\delta(k,\beta)/2$.
\end{rk}

Since $\delta(k,\beta)=\lng(\snell(O,(2k+3,1)))-\lng(\snell(O,(2k+1,1))-2$,
and it is clear that $\lng(\snell(O,(2k+3,1)))-\lng(\snell(O,(2k+1,1))\sim \beta+1$ for $k\to +\infty$,
one expects that $\delta(k,\beta)\sim \beta-1$. A more precise result is the following.

\begin{theorem}\label{t:dmon}
Let $\beta>1$ be fixed. Then
${(\delta(k, \beta))}_k$ is a strictly monotone increasing sequence and
\begin{equation}\label{f:dbeh}
\delta(k, \beta)=(\beta-1)-\frac{\beta}{2(\beta+1)} \frac{1}{k^2}+ O\left(\frac{1}{k^3}\right)\,,
\qquad k \to +\infty\,.
\end{equation}
\end{theorem}

\begin{proof}
We can define $\tau$, $\sn$, and $\delta$ respectively
through (\ref{f:vinca}), (\ref{f:vincb}) and (\ref{f:delta})
as smooth functions of $k\in\R$, $k\geq 0$.
Differentiating $\delta$ w.r.t. $k$, we get
\[
\begin{split}
\delta_k(k,\beta) = {} &\rad[\tau]{1}+\radc[\tau]{(k+2)\tau_k\tau}+
\radb[\tau]{\beta^2}+\radbc[\tau]{(k+1)\beta^2\tau_k\tau} \\
& - \rad[\sn]{1}-\radc[\sn]{(k+1)\sn_k\sn}-
\radb[\sn]{\beta^2}-\radbc[\sn]{k\beta^2\sn_k\sn}\,.
\end{split}
\]
On the other hand, differentiating the constraints (\ref{f:vinca}) and (\ref{f:vincb})
we obtain the identities
\begin{gather}
\radc[\tau]{(k+2)\tau_k}+\radbc[\tau]{(k+1)\beta^2 \tau_k}=-\rad[\tau]{\tau}-\radb[\tau]{\tau}\,, \label{f:diffvinca}\\
\radc[\sn]{(k+1)\sn_k}+\radbc[\sn]{k\beta^2 \sn_k}=-\rad[\sn]{\sn}-\radb[\sn]{\sn} \,. \label{f:diffvincb}
\end{gather}
Hence, being $\tau<\sn$, we get
\[
\delta_k(k, \beta)=\sqrt{1-\tau^2}-\sqrt{1-\sn^2}+\sqrt{\beta^2-\tau^2}-\sqrt{\beta^2-\sn^2}>0\,.
\]

In order to determine the behavior of $\delta(k,\beta)$ for $k$ large, notice that, setting $\varepsilon=1/k$ and
$\widetilde{\sn}(\varepsilon)= \sn(1/\varepsilon,\beta)$, (\ref{f:vincb}) becomes
\[
\rad[\widetilde{\sn}]{(1+\varepsilon)\widetilde{\sn}}+\radb[\widetilde{\sn}]{\widetilde{\sn}}=\varepsilon\,,
\]
that is $\widetilde{\sn}$ is implicitly defined by
\[
f(\widetilde{\sn})-\varepsilon=0\,, \qquad
f(t)=\frac{\dfrac{t}{\sqrt{1-t^2}}+\dfrac{t}{\sqrt{\beta^2-t^2}}}{1-\dfrac{t}{\sqrt{1-t^2}}}\,.
\]
One has $f(0)=0$, $f'(0)= \dfrac{\beta+1}{\beta}$, $f''(0)= 2\dfrac{\beta+1}{\beta}$, so that
\[
\widetilde{\sn}(0)=0,\quad \widetilde{\sn}'(0)=\frac{\beta}{\beta+1}, \quad
\widetilde{\sn}''(0)=-2 \frac{\beta^2}{(\beta+1)^2}\,,
\]
and hence
\begin{equation}\label{f:asintotico2}
\widetilde{\sn}(\varepsilon)=\frac{\beta}{\beta+1}\varepsilon -\frac{\beta^2}{(\beta+1)^2} \varepsilon^2+
O(\varepsilon^3) \qquad \varepsilon \to 0^+\,,
\end{equation}
that is
\begin{equation}\label{f:asintotico1}
\sn(k,\beta) =
\frac{\beta}{\beta+1}\,\frac{1}{k}-
\frac{\beta^2}{(\beta+1)^2}\frac{1}{k^2}+ O\left(\frac{1}{k^3}\right)\,,
\end{equation}
and
\begin{equation}\label{f:asintotico}
\begin{aligned}
\tau(k,\beta) =
\sn(k+1,\beta)&=\widetilde{\sn}\left(\frac{1}{k+1}\right)=\widetilde{\sn}\left(\frac{1}{k}-\frac{1}{k^2}
+O\Big(\frac{1}{k^3}\Big)\right)
\\
&=\frac{\beta}{\beta+1}\,\frac{1}{k}-
\left(\frac{\beta}{\beta+1}+\frac{\beta^2}{(\beta+1)^2}\right)\frac{1}{k^2}+ O\left(\frac{1}{k^3}\right)\,.
\end{aligned}
\end{equation}
Finally we have
\begin{equation*}
\begin{aligned}
\delta(k, \beta)& = (k+2)\left(1+\frac{1}{2} \tau^2\right) +\beta(k+1) \left(1+\frac{\tau^2}{2\beta^2} \right)
-(k+1)\left(1+\frac{1}{2} \sn^2\right)
\\
&\quad- \beta k\left(1+\frac{\sn^2}{2\beta^2} \right)-2+O\left(\frac{1}{k^3}\right)
= (\beta-1)-\frac{\beta}{2(\beta+1)} \frac{1}{k^2}+ O\left(\frac{1}{k^3}\right)\,,
\end{aligned}
\end{equation*}
completing the proof.
\end{proof}

\begin{dhef}\label{d:kc}
Given $\beta>1$, we shall denote by $\kc(\beta)$ the
integer number defined by
\begin{equation}\label{f:kc}
\kc(\beta)=\min\{k\in\N\colon\ \delta(k,\beta) > 0\}\,.
\end{equation}
\end{dhef}

By the very definition, we have that $\ls(2\kc(\beta),\beta)\leq \ls(2k,\beta)$ for every $k\in \N$
(see also Remark \ref{r:rdd}).
Moreover, by Lemma~\ref{r:spiegaz1} and
Theorem~\ref{t:t1}, the absolute minimum
of $\ls(\cdot,\beta)$ is attained at a point $t=2k$, $k\in\N$.
Hence $\ls(2\kc(\beta),\beta)$ (i.e., the normalized length of the
Snell path joining the origin with the right-top vertex of the $(2\kc+1)$--th square)
minimizes
the normalized length $\ls(\cdot, \beta)$ among all the paths remaining in a single horizontal strip.

Now we want to study $\kc(\beta)$, $\beta>1$.
As a preliminary step we investigate the behavior of $\delta(k,\cdot)$ for a given $k$.

\begin{lemma}\label{l:dwrtb}
Let $k\in\N$ be fixed. Then the function $\beta\mapsto\delta(k,\beta)$ is  strictly increasing
in $[1,+\infty)$.
Moreover
$\delta(k,1) < 0$
and $\delta(k,\sqrt{2})>0$.
\end{lemma}

\begin{proof}
Differentiating $\delta$ w.r.t. $\beta$, we get
\begin{equation}\label{f:dwrtb1}
\begin{split}
\delta_\beta(k,\beta)& = \radc[\tau]{(k+2) \tau \tau_\beta}+(k+1)\beta\radbc[\tau]{\beta^2-2\tau^2}+
(k+1)\radbc[\tau]{\beta^2 \tau \tau_\beta} \\
& -\radc[\sn]{(k+1) \sn \sn_\beta}-k\beta\radbc[\sn]{\beta^2-2\sn^2}-
k\radbc[\sn]{\beta^2 \sn \sn_\beta}\,.
\end{split}
\end{equation}
Differentiating (\ref{f:vinca}) and (\ref{f:vincb}) w.r.t. $\beta$, we obtain
\begin{gather}
\radc[\tau]{(k+2) \tau \tau_\beta}+(k+1)\radbc[\tau]{\beta^2 \tau \tau_\beta}=
(k+1)\radbc[\tau]{\tau^2\beta} \label{f:dvin1} \\
\radc[\sn]{(k+1) \sn \sn_\beta}+ k\radbc[\sn]{\beta^2 \sn \sn_\beta}=k\radbc[\sn]{\tau^2\beta}\,. \label{f:dvin2}
\end{gather}
Substituting (\ref{f:dvin1}) and (\ref{f:dvin2}) into (\ref{f:dwrtb1}), we conclude that
\[
\delta_\beta(k,\beta)=\radb[\tau]{(k+1)\beta}-\radb[\sn]{k\beta}\,.
\]

We have to show that
\begin{equation}\label{f:dwrtb2}
\radb[\tau]{(k+1)\beta}-\radb[\sn]{k\beta}>0, \qquad \forall k\in\N.
\end{equation}
For every $\kappa\in \R$, let us denote by $s(\kappa)$ the unique function implicitly defined
by 
\begin{equation}\label{f:vinkap}
\vincolo[s]{(\kappa+1)}{\kappa}=1
\end{equation}
so that $s(k)=\sn$, $s(k+1)=\tau$. Since inequality (\ref{f:dwrtb2}) clearly holds true for $k=0$,
it is enough to show that
\[
\frac{d}{d \kappa}\left(\radb[s(\kappa)]{\kappa}\right)=
\radbc[s]{\beta^2-s^2+\kappa s s_\kappa}>0\,\qquad \forall \kappa\in\R, \ \kappa\geq 1.
\]
Differentiating (\ref{f:vinkap}) w.r.t. $\kappa$ (see also (\ref{f:diffvinca})), we get
\[
s_\kappa(\kappa)=-s \frac{\rad[s]{1}+\radb[s]{1}}{\radc[s]{\kappa+1}+\radbc[s]{\kappa\beta^2}}\,.
\]
Moreover, using again (\ref{f:vinkap}), we have
\[
\rad[s]{1}+\radb[s]{1}=\frac{1}{\kappa s}-\frac{1}{\kappa \sqrt{1-s^2}}< \frac{1}{\kappa s}\,,
\]
so that
\[
\begin{split}
\beta^2-s^2+\kappa s s_\kappa & =\beta^2-s^2-\kappa s^2 \frac{\rad[s]{1}+\radb[s]{1}}{\radc[s]{\kappa+1}+\radbc[s]{\kappa\beta^2}}\\
& >
\frac{\beta^2-s^2}{\kappa\beta^2}\left(\kappa\beta^2-s\sqrt{\beta^2-s^2}\right)>0\,,
\end{split}
\]
where the last inequality can be easily checked recalling that $0< s <1$,
while $\beta$, $\kappa\geq 1$.
Hence we conclude that inequality (\ref{f:dwrtb2})
holds true, which implies that the function $\delta(k,\beta)$ is  strictly
increasing w.r.t. $\beta$.

Taking into account that, by \eqref{f:vincb}, $\sn(0,\sqrt 2) = \sqrt 2/2$,
we easily get that $\delta(0,\sqrt 2)>0$, so that, by Theorem \ref{t:dmon},
$\delta(k, \sqrt{2})>0$.

In order to prove that
$\delta(k,1) < 0$,
we note that, from
(\ref{f:vinca}) and (\ref{f:vincb}), we get
\[
\tau(k,1)=\frac{1}{\sqrt{1+(2k+3)^2}},\qquad
\sn(k,1)=\frac{1}{\sqrt{1+(2k+1)^2}}\,,
\]
so that
\[
\delta(k,1)=\sqrt{1+(2k+3)^2}-\sqrt{1+(2k+1)^2}-2<0\,,
\]
concluding the proof.
%
%
%
\end{proof}

\begin{rk}\label{r:r3}
As an easy consequence of Lemma~\ref{l:dwrtb} we obtain that $\kc(\beta)$ is
a nonincreasing function of $\beta$.
\end{rk}

Thanks to Lemma \ref{l:dwrtb}, the following definition makes sense.

\begin{dhef}\label{d:bc}
For every $k\in\N$ we shall denote by $\bc_k$  the unique number in
$(1, \sqrt{2})$ such that
$\delta(k, \bc_k)=0$.
\end{dhef}

By Theorem~\ref{t:dmon} we have that
\[
\delta(j, \bc_k)<0, \quad \forall\ j<k, \qquad
\delta(j, \bc_k)>0, \quad \forall\ j>k,
\]
hence
\begin{equation}\label{f:minbk}
\ls(2k,\bc_k)=\ls(2k+2,\bc_k)<\ls(2j,\bc_k) \quad \forall\ j\in\N\setminus \{k,\ k+1\} \,.
\end{equation}
In particular we have that
\begin{equation}\label{f:lbcneg}
\ls(2k+2,\bc_k)=\ls(2k,\bc_k)\leq \ls(0,\bc_k)=0\,,
\end{equation}
where we have taken into account that $\ls(0,\beta)=0$, for every $\beta\geq 1$.

\begin{lemma}\label{l:bcdecr}
The sequence ${(\bc_k)}_{k\in\N}$ is strictly decreasing and
$\lim_{k\to +\infty}\bc_k=1$.
\end{lemma}

\begin{proof}
Given $k\in \N$, by the monotonicity of $\delta$ w.r.t. $k$ stated in
Theorem \ref{t:dmon} we have
\[
\delta (k, \bc_{k+1})<\delta (k+1, \bc_{k+1})=0\,,
\]
where the last equality follows from the very definition of $\bc_{k+1}$.
Again, the definition of $\bc_k$ and the monotonicity of $\delta$ w.r.t. $\beta$
stated in Lemma \ref{l:dwrtb} imply that $\bc_{k+1}<\bc_k$.

In order to prove the last part of the thesis, given $k\in\N$, let us define the functions
\begin{equation}\label{f:deffg}
\begin{split}
f^k(s,\beta) & :=\lungh[s]{(k+1)}{k}-2k-\sqd\,,\\
g^k(s,\beta) & :=\vincolo[s]{(k+1)}{k}-1\,,\\
h^k(s,\beta) & :=f^k(s,\beta)-s\, g^k(s,\beta)
=(k+1)\sqrt{1-s^2}+k\sqrt{\beta^2-s^2}-2k-\sqd+s\,,
\end{split}
\end{equation}
where $s\in [0,1)$ and $\beta >1$.
Since $\hsn(2k,\beta)$ is the unique solution of $g^k(s,\beta)=0$,
we have that $h^k(\hsn(2k,\beta),\beta)=f^k(\hsn(2k,\beta),\beta)=\ls(2k,\beta)$.
Moreover $h_s^k(s,\beta)=-g^k(s,\beta)$ and $g_s^k(s,\beta)>0$,
so that $h^k(\cdot,\beta)$ is a strictly concave function in [0,1] which attains
its absolute maximum at $s=\hsn(2k,\beta)$.
Hence
\[
\ls(2k,\beta)=h^k(\hsn(2k,\beta),\beta)=\max_{s\in [0,1)}h^k(s,\beta)
> h^k(0,\beta) = (\beta-1)k+1-\sqrt{2}\,,
\quad\forall \beta>1\,.
\]
Then, from (\ref{f:lbcneg}), we have
\[
0\geq \ls(2k+2,\bc_k)> (\bc_k-1)(k+1)+1-\sqrt{2}
\]
so that
\begin{equation}\label{f:estbck}
1< \bc_k < 1+\frac{\sqd-1}{k+1}
\end{equation}
and the conclusion follows.
\end{proof}

The first values $(\bc_k)$ (up to the fifth digit)
are listed in the following table.

\smallskip
\begin{tabular}{|c||c|c|c|c|c|c|c|c|}
\hline
$k$ & 0 & 1 & 2 & 3 & 4 & 5 & 6 & 7 \\
\hline
$\bc_k$ &
1.24084 &
1.06413 &
1.02820 &
1.01577 &
1.01006 &
1.00698 &
1.00512 &
1.00392 \\
\hline
\end{tabular}

\bigskip

By Lemma \ref{r:spiegaz1}(v), for $\beta\geq \sqrt{3/2}$ the Snell path $S(O,(t,1))$, with $t\in(0,1)$
is never a geodesic. This property will be crucial for the results in Section \ref{s:geovere}.
The following result gives the position of $\sqrt{3/2}\simeq 1.22474$ w.r.t. the critical values $\bc_k$.

\begin{lemma}
$\bc_1 < \sqrt{3/2} < \bc_0$.
\end{lemma}

\begin{proof}
From (\ref{f:estbck}) we have
\[
\bc_1< 1+\frac{\sqrt{2}-1}{2} <  \sqrt{\frac{3}{2}},\qquad \bc_0<\sqd\,.
\]
In order to prove the inequality $\sqrt{3/2} < \bc_0$,
let $f^1$, $g^1$, $h^1$ be the functions defined in (\ref{f:deffg}) for $k=1$,
and
let $\sn_1 \in (0,1)$ be the unique solution of $g^1(s,\bc_0)=0$.
Since
\[
0=\delta(0,\bc_0)=\ls(2,\bc_0)=f^1(\sn_1,\bc_0)\,,
\]
we have that $h^1(\sn_1,\bc_0)=0$. Moreover
\[
h^1_\beta(s,\beta)=\radb[s]{\beta}>0\,.
\]
Hence the inequality $\sqrt{3/2}<\bc_0$ can be obtained showing that
\begin{equation}\label{f:hatm}
h^1\left(s,\sqrt{\frac{3}{2}}\right)=2\sqrt{1-s^2}+\sqrt{\frac{3}{2}-s^2}-2-\sqd+s<0\,,
\qquad \forall s\in (0,1).
\end{equation}
The inequality (\ref{f:hatm}) easily follows observing that
\[
2\sqrt{1-s^2}+\frac{2}{3}s \leq \frac{2}{3}\,\sqrt{10}, \quad
\sqrt{\frac{3}{2}-s^2}+\frac{1}{3}s \leq \sqrt{\frac{5}{3}},\qquad \forall s\in (0,1)\,,
\]
so that
\[
h^1\left(s,\sqrt{\frac{3}{2}}\right)\leq\frac{2}{3}\,\sqrt{10}+\sqrt{\frac{5}{3}}-2-\sqd<0,
\qquad \forall s\in (0,1)\,,
\]
which completes the proof.
\end{proof}

\begin{figure}
\includegraphics[height=4cm]{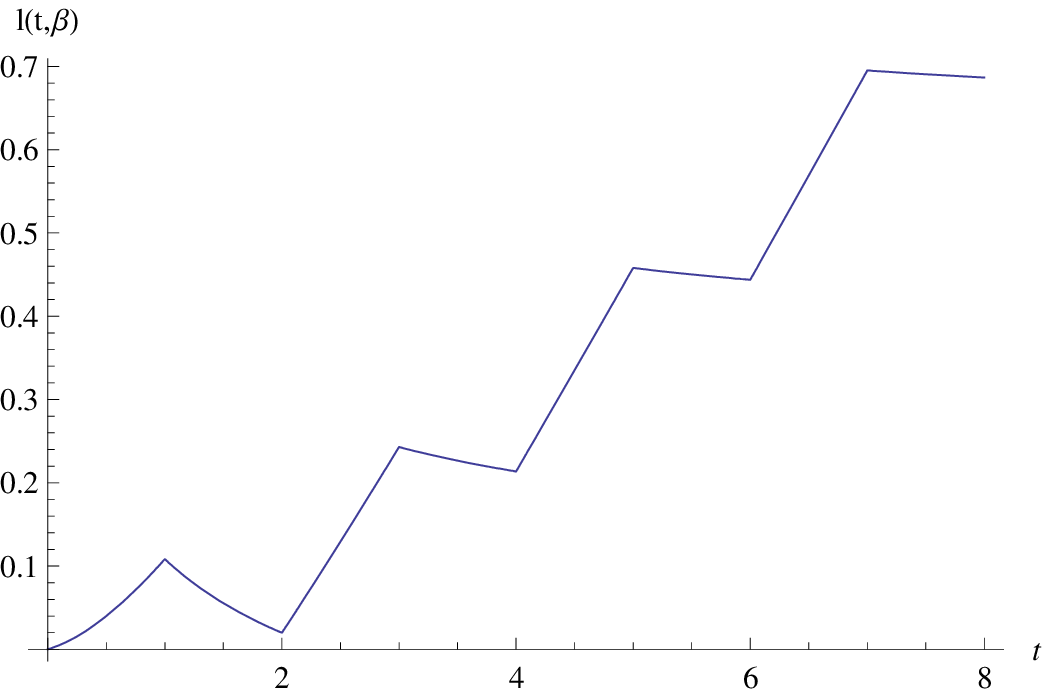} \qquad
\includegraphics[height=4cm]{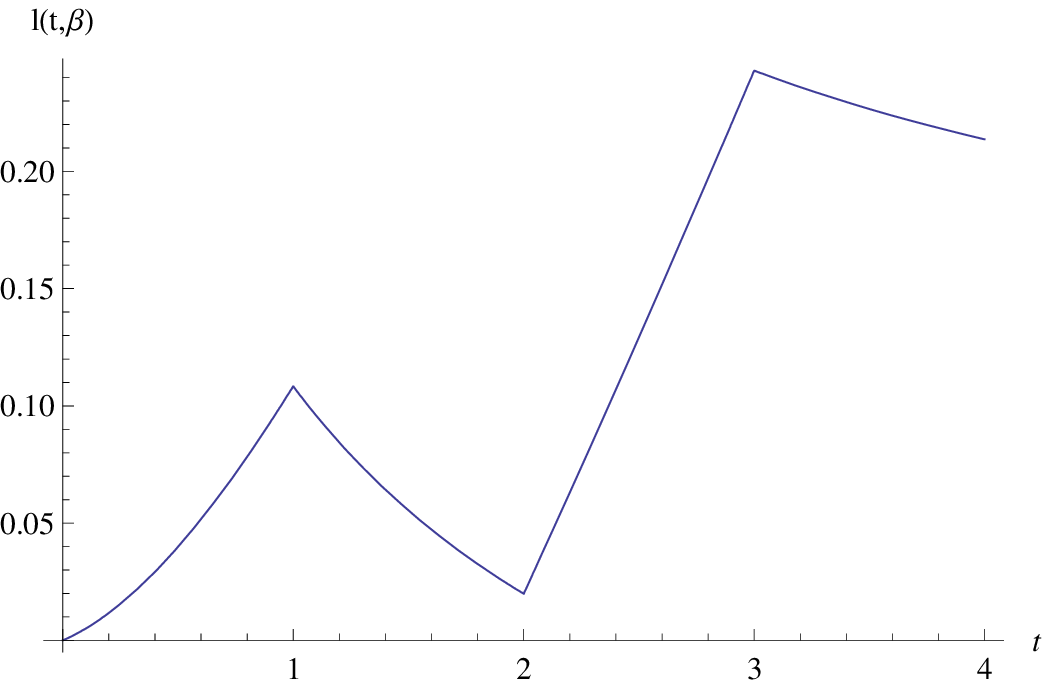}
\caption{Plot of $\ls(t,\beta)$, $\beta=1.26$ ($\beta>\bc_0$)}
\label{figp3}
\end{figure}

\begin{figure}
\includegraphics[height=4cm]{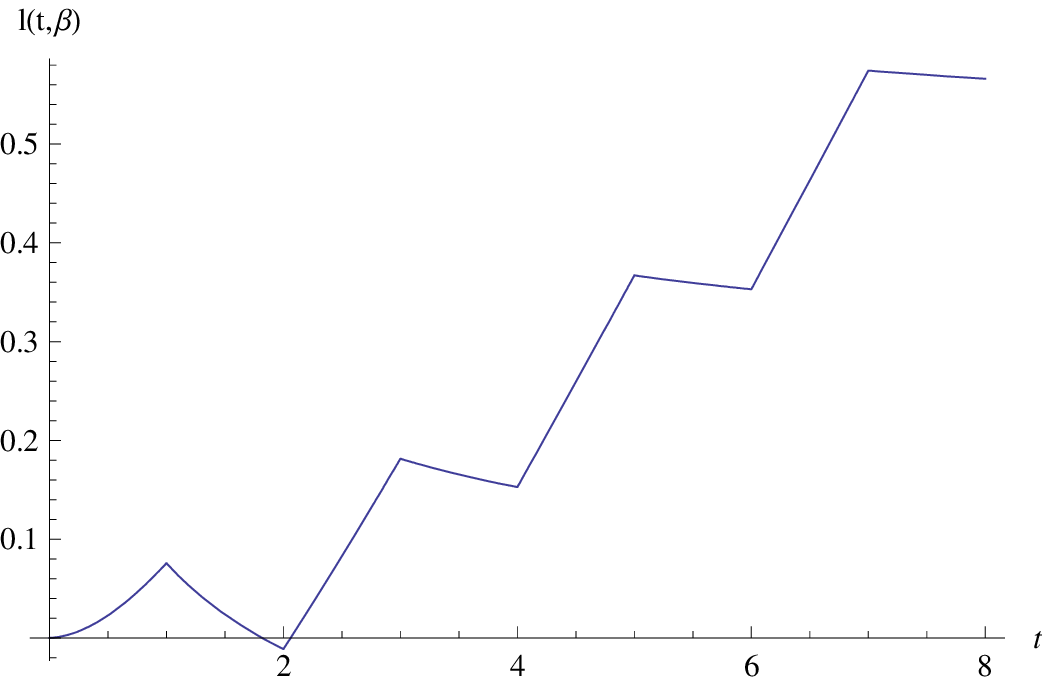} \qquad
\includegraphics[height=4cm]{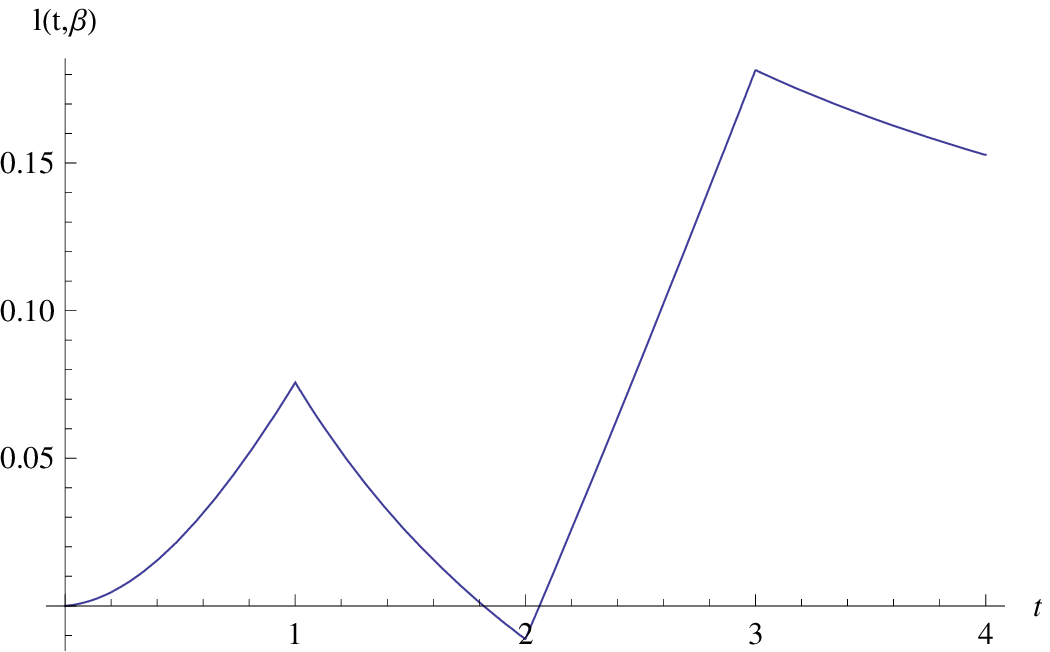}
\caption{Plot of $\ls(t,\beta)$, $\beta=1.23$ ($\sqrt{3/2}<\beta<\bc_0$)}
\label{figp2}
\end{figure}

\begin{figure}
\includegraphics[height=4cm]{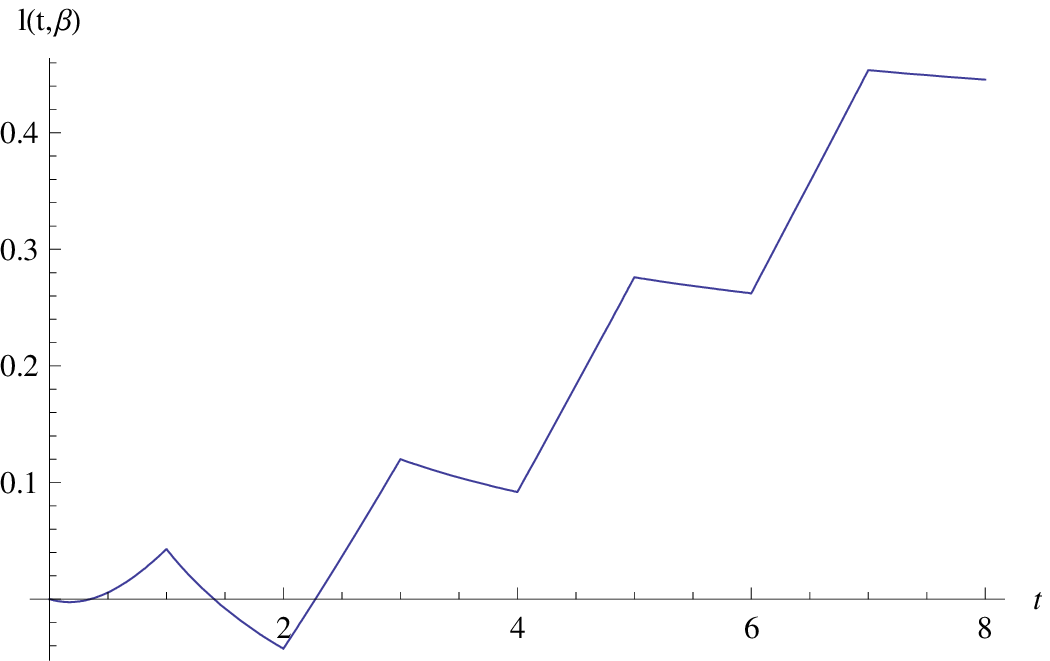} \qquad
\includegraphics[height=4cm]{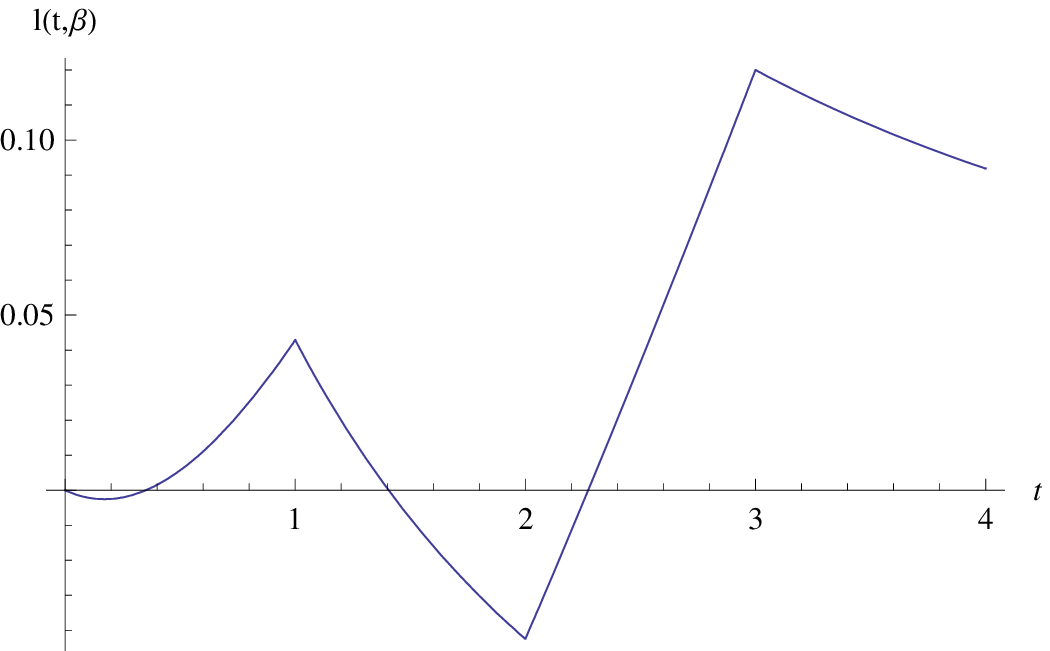}
\caption{Plot of $\ls(t, \beta)$, $\beta=1.2$ ($\bc_1<\beta<\sqrt{3/2}$)}
\label{figp1}
\end{figure}

We summarize the previous analysis in the following result, which is depicted
in Figures~\ref{figp3}--\ref{figp4}.

\begin{cor}\label{c:minb}
For every $\beta>1$
\begin{equation}\label{f:minb}
\min_{s\geq 0}\ls(s, \beta)=\min_{k\in \N}\ls(2k,\beta)<\ls(t,\beta) \qquad \forall\ t\in \bigcup_{k\in\N}(2k, 2k+2).
\end{equation}
Moreover the following hold.
\begin{itemize}
\item[(i)] $k_c(\beta)=0$ for every $\beta> \bc_0$, and $k_c(\beta)=k+1$ for every $\beta \in (\bc_{k+1},\bc_k]$, $k\in\N$;
\item[(ii)] if $\beta>\bc_0$, then $0=\ls(0,\beta)<\ls(t,\beta)$ for every $t>0$;
\item[(iii)] if $\beta \in (\bc_{k+1},\bc_k)$, $k\in\N$, then
$\ls(2k+2,\beta)<\ls(t,\beta)$ for all $t\in [0,+\infty)\setminus \{2k+2\}$ ;
\item[(iv)] $\ls(2k,\bc_k)=\ls(2k+2,\bc_k)<\ls(t, \bc_k)$ for every $k\in\N$, and
$t\in [0,+\infty)\setminus \{2k, 2k+2\}$;
\item[(v)] if $\beta \geq \sqrt{3/2}$, then $\ls(t,\beta)>0$ for $t\in (0, 1]$.
\item[(vi)]
for every $\beta>1$,
$\ls(2k_c(\beta),\beta) < \ls(t,\beta)$ for every $t>2 k_c(\beta)$.
\end{itemize}
\end{cor}

\begin{proof}

Formula (\ref{f:minb}) summarizes the results in Lemma \ref{r:spiegaz1} and Theorem \ref{t:t1}.

Let us now prove (i). By
Lemma \ref{l:dwrtb} and Theorem \ref{t:dmon} we have
\begin{alignat}{2}
\delta(j,\beta)\leq \delta(k,\beta)\leq\delta(k,\bc_k)=0\,, \qquad \forall j\in\N,\ j\leq k\,\ \beta\leq\bc_k,
& \label{f:disug1}
\\
\delta(j,\beta)\geq \delta(k+1,\beta)>\delta(k+1,\bc_{k+1})=0\,,
\qquad \forall j\in\N,\ j\geq k+1, \beta>\bc_{k+1} \,,
& \label{f:disug2}
\end{alignat}
and we have the strict inequality in (\ref{f:disug1}) if $\beta\neq \bc_k$.
Hence, recalling the definition of $\kc(\beta)$ given in \eqref{f:kc}, we conclude that (i) holds.

Moreover, since for every $n$, $m\in\N$, $n < m$, we have
\[
\ls(2m,\beta)=\ls(2n,\beta)+\sum_{j=n}^{m-1} \delta(j,\beta)\,,
\]
as a consequence of \eqref{f:disug1} and \eqref{f:disug2}, for every $k\in\N$ we get
\begin{gather}
0=\ls(0,\beta)<\ls(2j,\beta), \qquad \forall j\in\N, \ j\neq 0,\ \beta>\bc_0 \\
\ls(2k+2,\beta)<\ls(2j,\beta), \qquad \forall j\in\N, \ j\neq k+1,\ \beta\in (\bc_{k+1}, \bc_k)\\
 \ls(2k,\bc_k)=\ls(2k+2,\bc_k)<\ls(2j, \bc_k), \qquad \forall j\in\N, \ j\neq k, k+1,\,.
\end{gather}
Hence (ii), (iii), and (iv) follow from (\ref{f:minb}).
Finally (v) follows from Lemma~\ref{r:spiegaz1}(v),
whereas (vi) is a direct consequence of (i)--(iv).
\end{proof}

\begin{figure}
\includegraphics[height=4cm]{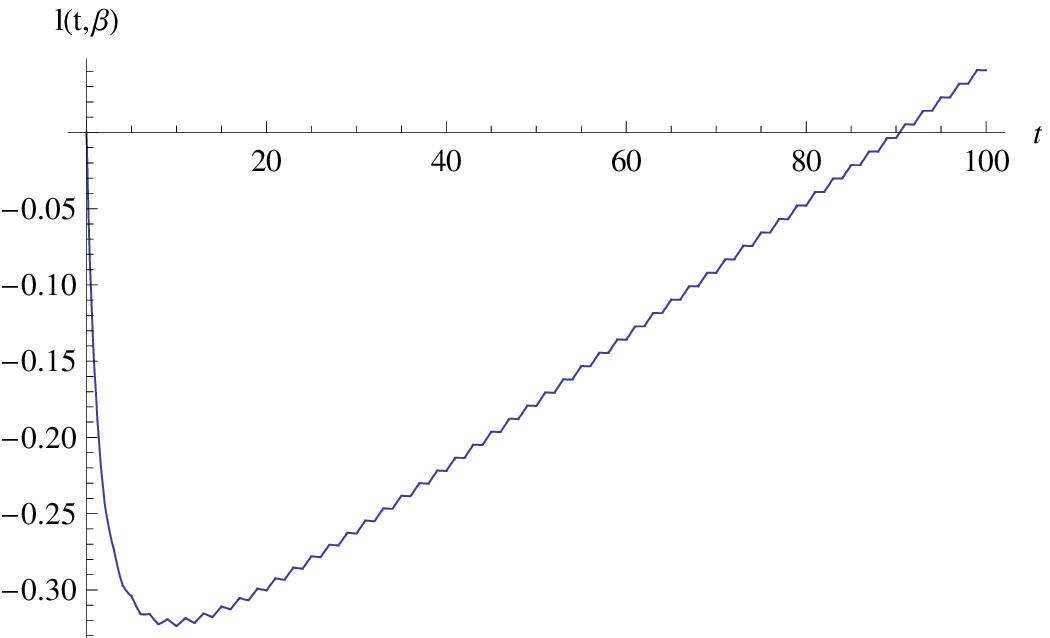} \qquad
\includegraphics[height=4cm]{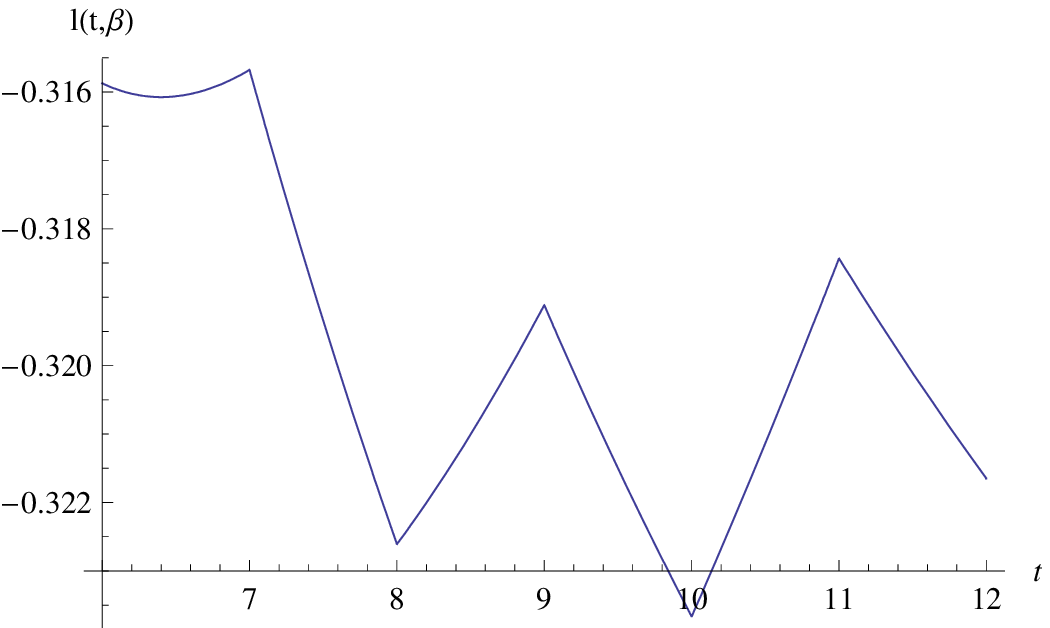}
\caption{Plot of $\ls(t, \beta)$, $\beta=1.009$ ($\bc_5<\beta<\bc_4$)}
\label{figp4}
\end{figure}

Up to now we have described the behavior of the normalized length of Snell paths starting
from a light vertex and remaining in a horizontal strip.
The following result deals with the normalized length of any Snell path remaining in a horizontal
strip.

\begin{lemma}\label{l:disnell}
Given $x\in \R$, $\tau\geq 0$, $r\in\Z$,
let $\snell(A,B)$ be the Snell path joining $A=(x,r)$ to $B=(x+\tau,r+1)$. Then
\begin{equation}\label{f:genineq}
\lng(\snell(A,B))-\tau+1-\sqrt{2}\geq \ls(2\kc(\beta),\beta)\,.
\end{equation}
Moreover
\begin{itemize}
\item[i)] if $\beta \neq \bc_{k}$ for every $k\in\N$, then the equality in (\ref{f:genineq}) holds
if and only if $\tau=2\kc(\beta)+1$, and $x=2n$, $n\in\Z$;
\item[ii)] if $\beta=\bc_{k}$ for some $k\in\N$, then the equality in (\ref{f:genineq}) holds
if and only if $\tau\in\{2\kc(\beta)-1,2\kc(\beta)+1\}$, and $x=2n$, $n\in\Z$.
\end{itemize}
\end{lemma}

\begin{proof}
By Remark \ref{r:thfix} we have that
\begin{equation}\label{f:genin1}
\lng(\snell(A,B)) \geq \lng(\snell(0,C))\,,\qquad C=(\tau,1)\,,
\end{equation}
since $\snell(0,C)$ crosses a quantity of dark material not greater then the one crossed by
any other Snell path with thickness $\tau$ .
On the other hand, if $\tau\in[0,1)$, then
\[
\lng(\snell(O,C))=\sqrt{1+\tau^2}>\tau-1+\sqrt{2}\,,
\]
so that, by (\ref{f:genin1}),
\[
\lng(\snell(A,B))-\tau+1-\sqrt{2}>0\geq \ls(2\kc(\beta)+2,\beta)\,.
\]
Moreover, we stress that the equality in (\ref{f:genineq}) never occurs when $\tau\in [0,1)$.
If $\tau\geq 1$, then
\[
\lng(\snell(0,C))= \ls(\tau,\beta)+\tau+\sqrt{2}\,,
\]
so that
\begin{equation}\label{f:genin2}
\lng(\snell(A,B))-\tau+1-\sqrt{2} \geq \ls(\tau,\beta)\geq \ls(2\kc(\beta),\beta)\,.
\end{equation}
It remains to discuss the occurrence of the equality in (\ref{f:genin2}) for $\tau\geq 1$.

If $\beta \neq \bc_{k}$, $k\in\N$, then $\ls(\cdot,\beta)$ has its strict absolute minimum
point at $t=2\kc(\beta)$, so that $\ls(\tau-1,\beta)=\ls(2\kc(\beta),\beta)$ if and only if $\tau=2\kc(\beta)+1$.

Moreover, for $\tau=2\kc(\beta)+1$, the equality $\lng(\snell(A,B))=\lng(\snell(0,C))$
holds if and only if $x=2n$, $n\in \Z$. Namely, if $p_x$ denotes the thickness of light
material crossed by $\snell(A,B)$, then $p_x\leq \kc(\beta)+1=p(2\kc(\beta)+1)$ (since the total
thickness is $2\kc(\beta)+1$),
and $p_x=\kc(\beta)+1$ if and only if $A$ is a light vertex (i.e.\ $x=2n$, $n\in \Z$).

If $\beta=\bc_k$ for some $k\in\N$, then the conclusion in (ii) follows from (\ref{f:genin2}) and the
fact that the absolute minimum of $\ls(\cdot,\bc_{k})$ is attained both for $t=2\kc(\beta)-1$, and
$t=2\kc(\beta)+1$.
\end{proof}

We conclude this section by stating some properties, which will be useful in the second part of Section \ref{s:geod},
of the following generalization of the normalized length introduced in \eqref{f:defelle}. Let
$q(t)$, $t\geq 0$, be the function defined in (\ref{f:defq}).
For $0 < h \leq 1$, let $p(t,h)=t+h-q(t)$.
Given $\beta>1$, let $\tsn(t,\beta,h)$ be the unique solution of the
implicit equation
\begin{equation}\label{f:vincth}
\vincolo[\tsn]{p(t,h)}{q(t)}-h=0\,,
\end{equation}
and let us define the function
\begin{equation}\label{f:elleh}
\lh(t,\beta,h)=\lungh[\tsn(t,\beta,h)]{p(t,h)}{q(t)}-t-h\sqd\,.
\end{equation}
The function $\lh$ is the normalized length of a Snell path starting from the point
$(-h,-h)$ and ending in $(t,0)$. It is straightforward that $\tsn(t,\beta,1)$ and $\lh(t,\beta,1)$
coincide
with the functions $\hsn(t,\beta)$ and $\ls(t,\beta)$ defined in (\ref{f:vinct}) and (\ref{f:defelle})
respectively. The function
$t \mapsto \lh(t,\beta,h)$ has the same qualitative properties of $\lh(\cdot, \beta,1)$
studied at the beginning of this Section. More precisely the derivative
\begin{equation}\label{f:ellet}
\lh_t(t,\beta,h)=
\begin{cases}
\sqrt{1-\tsn^2(t,\beta,h)}-1, & \text{if\ } t\in I_L, \\
\sqrt{\beta^2-\tsn^2(t,\beta,h)}-1, & \text{if\ } t\in I_D,
\end{cases}
\end{equation}
is a monotone increasing function both in the set $I_L$ and in the set
$I_D$ (see Lemma \ref{l:lpq}).

Concerning the derivative $\lh_h$ w.r.t. $h$, we have the following result.

\begin{lemma}\label{l:elleh}
The function $h\mapsto \lh(t,\beta,h)$ is strictly convex and monotone decreasing in $(0,1]$
for every $t>0$, $\beta>1$, and
\begin{equation}\label{f:delleh}
\lh_h(t,\beta,h)= \sqrt{1-\tsn^2(t,\beta,h)}+\tsn(t,\beta,h)-\sqd.
\end{equation}
\end{lemma}

\begin{proof}
Since $p_h(t,h)=1$ we have that
\begin{equation}\label{f:derh}
\lh_h(t,\beta,h)=
\frac{1}{\sqrt{1-\tsn^2}}+\frac{p\tsn\tsn_h}{(1-\tsn^2)^{3/2}}+
\frac{q\beta^2\tsn\tsn_h}{(\beta^2-\tsn^2)^{3/2}}-\sqd\,.
\end{equation}
Differentiating (\ref{f:vincth}) w.r.t. $h$, we get
\begin{equation}\label{f:deriv1}
\frac{p\,\tsn_h}{(1-\tsn^2)^{3/2}}
+\frac{\beta^2\, q\, \tsn_h}{(\beta^2-\tsn^2)^{3/2}}=1-\frac{\tsn}{\sqrt{1-\tsn^2}}\,.
\end{equation}
Substituting \eqref{f:deriv1} in (\ref{f:derh}), we obtain (\ref{f:delleh}).
It is straightforward to check that the function
$\tsn \mapsto \sqrt{1-\tsn^2}+\tsn$ is strictly increasing in $(0,1/\sqd)$.
Therefore, since $t>0$ implies $\tsn<1/\sqrt 2$, it follows that $\lh_h(t,\beta,h)<0$, for $h\in (0,1)$.
Moreover, by \eqref{f:deriv1} it follows that $\tsn_h >0$, so that
the function $h\mapsto \tsn(t,\beta,h)$ is strictly increasing in $(0,1)$ for every $t>0$ and $\beta>1$.
Hence, $\lh_h$ is strictly increasing for $h\in(0,1)$, which implies that $\lh(t,\beta,h)$ is
strictly convex w.r.t. $h$.
\end{proof}

\section{General properties of the geodesics in the chessboard structure}\label{s:geod}

\begin{figure}
\includegraphics[height=4cm]{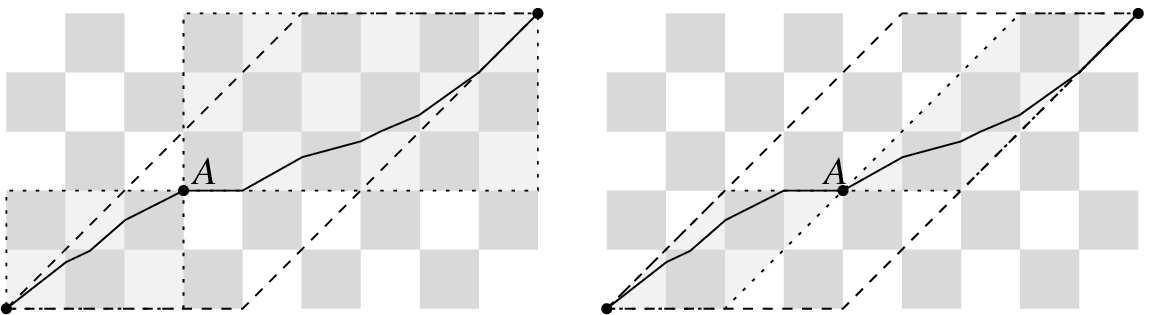}
\caption{}
\label{figd2}
\end{figure}

Up to now we have investigated the properties of a geodesic joining two points on the
sides of one horizontal strip in the chessboard structure. In this section we will
study the properties of a geodesic starting from the origin $O$, crossing
an arbitrary large number of horizontal strips, and ending in a light vertex
$(2n+j,j)$, $n,j\in\N$. If $n=0$ or $j=0$ , then the unique geodesic from the origin to
the point $(2n+j,j)$ is the segment joining the two points. Hence we shall further assume
that $j,n\geq 1$.

Throughout this section we shall assume that
\begin{equation}\label{f:gamma}
\text{$\Gamma$ is a geodesic from the origin to the point $(2n+j,j)$, $n,j\in\N$, $n,j\geq 1$}.
\end{equation}


The basic properties of $\Gamma$ are listed in the following two propositions.

\begin{prop}\label{p:bounds}
Let $\Gamma$ be as in (\ref{f:gamma}). Then the following properties hold.
\begin{itemize}
\item[(i)] Let $H$ be a closed half plane such that $\partial H$ is
either the line $x=k$, or  $y=k$, or a light diagonal $D_k$, for some $k\in\Z$.
Let $\widetilde{\Gamma}\subseteq \Gamma$
be a path, with endpoints $\widetilde{A}$, $\widetilde{B}$, such that $\widetilde{\Gamma}\subseteq H$
and $\widetilde{A}$, $\widetilde{B}\in\partial H$. Then $\widetilde{\Gamma}=\cray{\widetilde{A},\widetilde{B}}$.
\item[(ii)]  Let $A=(x_A,y_A)\in \Gamma$ and let $\Gamma^-$, $\Gamma^+ \subseteq \Gamma$ be the two paths
joining $O$ to $A$ and $A$ to $(2n+j,j)$ respectively. Then the following bounds hold:
\begin{itemize}
\item if $A \in \Z \times \Z$ then $\Gamma^-\subseteq [0,x_A]\times [0,y_A]$ and
$\Gamma^+\subseteq [x_A, 2n+j]\times [y_A,j]$;
\item if in addition $A$ is a light vertex, then
\[
\begin{split}
\Gamma^-  & \subseteq \{ (x,y)\in \R^2 \colon\ 0\leq y \leq y_A,\ y \leq x \leq y-y_A+x_A\}, \\
\Gamma^+  & \subseteq \{ (x,y)\in \R^2 \colon\ y_A\leq y \leq j,\ y-y_A+x_A \leq x \leq y+2n\}
\end{split}
\]
(see Figure \ref{figd2}).
\end{itemize}
\item[(iii)] Let $Q$ be the interior of a light or a dark square, and assume that
$\Gamma\cap Q\neq \emptyset$. Then $\Gamma\cap Q$ is a segment.
As a consequence, $\Gamma = \cup_{i=1}^N \cray{P_{i-1}, P_i}$, where
$P_0=(0,0)$, $P_N = (2n+j,j)$, $P_i\neq P_k$ for $i \neq k$, and $P_i$ belongs to the boundary
of a square for every $i=0,\ldots, N$.
\item[(iv)] Let $\theta_i$ denote the oriented angle between the horizontal axis and $\cray{P_{i-1}, P_i}$,
$i=1,\ldots, N$. Then $\theta_i\in [0, \pi/2]$. Moreover, if either $P_{i-1}$ or $P_i$ is a light
vertex, then $\theta_i\in[0,\pi/4]$.
\end{itemize}
\end{prop}

\begin{proof}
Property (i) is a straightforward consequence of the local minimality of $\Gamma$ and the fact that a segment
$\cray{\widetilde{A},\widetilde{B}}$ contained in the lines $x=k$, or  $y=k$, or $D_k$, for some $k\in\Z$ is the unique
geodesic from $\widetilde{A}$ to $\widetilde{B}$.

In order to prove (ii), we first observe that
\[
\Gamma  \subseteq \mathcal{W}\doteq \{ (x,y)\in \R^2 \colon\ 0\leq y \leq j,\ y \leq x \leq 2n+y\}\,.
\]
Namely, if this is not the case, there exists an open half plane $H$ such that $H\cap \mathcal{W} = \emptyset$,
$H \cap \Gamma \neq \emptyset$, and $\partial H$ is
one of the lines $y=0$, $y=j$, $D_0$, $D_n$, a contradiction with (i).
Then, if  $A\in \Gamma \subseteq \mathcal{W}$ has the stated requirements, then the bounds in (ii)
can be obtained reasoning as above with half planes with boundary given by
a line of the type $x=x_A$, or  $y=y_A$, or $D_{(x_A-y_A)/2}$.

Property (iii) is a necessary condition for minimality, see, e.g., \cite[Section IV]{AcBu}.

In order to prove (iv), we notice that, by (ii), $\theta_i\in [0,\pi/2]$ whenever  either $P_{i-1}$ or $P_i$ is in $\Z \times \Z$,
and  $\theta_i\in[0,\pi/4]$ if either $P_{i-1}$ or $P_i$ is a light vertex. In particular $\theta_1$, $\theta_N \in [0, \pi/4]$.
Hence we have only to show that if $\theta_{i-1}\in [0, \pi/2]$ and $P_{i-1} \not\in \Z \times \Z$, then $\theta_{i}\in [0, \pi/2]$.
This follows from the fact that $P_{i-1}$ is in the interior of a side of a square, so that the Snell's law (\ref{f:snell}) holds.
\end{proof}

\begin{rk}\label{r:seguno}
Since ${a_{\beta}}=1$ on the boundary of the squares, then Proposition \ref{p:bounds}(iii) can be improved
observing that the intersection of $\Gamma$ with the closure of a light
square is a segment.
\end{rk}

In what follows we will be interested in the intersections $\Gamma \cap D_k$, $k=1,\ldots, n$.

\begin{prop}\label{p:akbk}
Let $\Gamma$ be as in (\ref{f:gamma}), and let $P_i$, $i=0,\ldots,N$ be as in Proposition \ref{p:bounds}(iii).
Then the following properties hold.
\begin{itemize}
\item[(i)] For every $k=0,\ldots,n$ there exist $0\leq \eta_k \leq \zeta_k \leq j$ such that
$\Gamma\cap D_k=\cray{A_k,B_k}$, $A_k=(2k+\eta_k,\eta_k)$, $B_k=(2k+\zeta_k,\zeta_k)$.
Moreover $A_0=(0,0)=P_0$, $B_n=(2n+j,j)=P_N$, and $\zeta_k\leq \eta_{k+1}$ for every $k=0,\ldots,n-1$.
\item[(ii)] If $A_k \neq B_k$ then $\eta_k$ and $\zeta_k$ are integers (that is $A_k$ and $B_k$
are light vertices).
\item[(iii)] Given $k=1,\ldots,n$, let $i=1,\ldots N$ be such that $A_k\in \ocray{P_{i-1}, P_i}$.
Then $0\leq \theta_i < \pi/4$. If in addition
$A_k$ is not a light vertex, then $\theta_i\neq 0$.
The same properties hold if $B_k\in \coray{P_{i-1}, P_i}$.
\end{itemize}
\end{prop}

\begin{proof}
It is clear that $\Gamma \cap D_k \neq \emptyset$ for every $k=1,\ldots, n$. Moreover, by
Proposition \ref{p:bounds}(i) and Remark \ref{r:seguno}, $\Gamma \cap D_k$ is either a single point, or a segment joining two light
vertices.
The inequality $\zeta_k\leq \eta_{k+1}$, $k=0,\ldots,n-1$ and property (ii) then follow from
Proposition \ref{p:bounds}(iv).


Let us now prove that (iii) holds. By Proposition \ref{p:bounds}(iv), we know that
$\theta_i \in [0,  \pi/2]$. Moreover, if $A_k$ is a light vertex, then
by Proposition \ref{p:bounds}(ii) with $A=A_k$ we get $\theta_i \in [0,  \pi/4]$,.
Finally, it has to be $\theta_i \neq \pi/4$ otherwise $\cray{P_{i-1},A_k}\subseteq \Gamma\cap D_k$,
in contradiction with the definition of $A_k$.

Assume now that $A_k$ belongs to the interior of a light square. Then, by (ii),
$\Gamma \cap D_k = \{A_k\}$, so that $\theta_i< \pi/4$.
Finally  $\theta_i\neq 0$, otherwise $\Gamma$ has to be an horizontal segment, due to Snell's Law (\ref{f:snell}).
\end{proof}

\begin{dhef}\label{d:cuts}
Given $k=0,\ldots,n$, we say that $\Gamma$ cuts the light diagonal $D_k$
if the points $A_k$ and $B_k$, defined in Proposition \ref{p:akbk},
coincide and belong to the interior of a light square.
\end{dhef}

For every $r=1,\ldots, j$, we shall denote by $\Gamma_r$ the curve
\begin{equation}\label{f:gammar}
\Gamma_r= \Gamma \cap \{r-1 <y <r\}\,.
\end{equation}
By Proposition \ref{p:bounds}(i), the intersection $A:=\overline{\Gamma}_r\cap\{y=r-1\}$
is a single point, as well as for
$B:=\overline{\Gamma}_r\cap\{y=r\}$.
The curve $\Gamma_r$ is a Snell path joining the two points $A$ and $B$,
and lying in a single horizontal strip.

\begin{rk}\label{r:esceino}
In what follows we shall assume, without loss of generality, that $\Gamma_1$ is a Snell path starting from
the origin. Namely, if this is not the case, $\Gamma_1=S(C_0,C_1)$ where $C_0=(x_0,0)$ and $C_1=(x_1,1)$,
$0<x_0<x_1$. Let us denote by $p_1$, $q_1$ respectively the thickness of the light zone and of the dark
zone crossed by $\Gamma_1$ , and by $p_2$, $q_2$ the analogous quantities for the
Snell path $S(O,C_2)$, $C_2=(x_1-x_2, 1)$. We have $p_1+q_1=p_2+q_2$, and $p_2\leq p_1$, so that, by
Remark \ref{r:thfix}, $\lng(\Gamma_1)\geq \lng(S(O,C_2))$. Hence
the curve $S(O,C_2)\cup [C_2,C_1]\cup\left(\Gamma \cap \{y\geq 1\}\right)$ is a geodesic.
\end{rk}

The following result is another fairly general property of the geodesics based
on the behavior of the function $\ls(t,\beta)$ studied in the previous section.

\begin{prop}\label{p:diag}
Let $\beta>1$ be given, and let $\kc(\beta)$ be as in Definition \ref{d:kc} (see also Corollary \ref{c:minb}(i)).
For every $r=1,\ldots, j$ the curve
$\Gamma_r$ defined in (\ref{f:gammar}) intersects at most $\kc(\beta)$ light diagonals.
\end{prop}

\begin{proof}
Set, as above,
$A=\overline{\Gamma}_r\cap\{y=r-1\}=(x_A,r-1)$
and $B=\overline{\Gamma}_r\cap\{y=r\}=(x_B,r)$
respectively the starting and the ending point of
the Snell path $\overline{\Gamma}_r$.

Assume by contradiction that $\Gamma_r$
intersects more than $\kc(\beta)$ light diagonals, so that $t_0:=x_B-x_A-1>2\kc(\beta)$.


Let $k\in\N$ be the smallest integer such that $2k+r-1\geq x_A$, and let
$m\in\N$ be the largest integer such that $2m+r\leq x_B$.
Set $A_0: = (2n+r-1, r-1)$, $B_0 := (2m+r,r)$.
It is clear that $A_0$ and $B_0$ lie respectively on the first and the last
light diagonal intersected by $\Gamma_r$
(see Figure~\ref{fig3}), so that, by assumption,
$m-n\geq \kc(\beta)$.

\begin{figure}
\includegraphics[width=8cm]{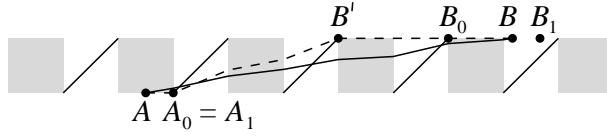}
\caption{Construction of $\widetilde{\Gamma}$, $\kc(\beta)=1$.}
\label{fig3}
\end{figure}

Let us denote by $\Delta_A = 2k+r-1-x_A$ and $\Delta_B = x_B-2m-r$.
It is not restrictive to assume that $\Delta_A\leq \Delta_B$.
The points $A_1=(\alpha,r-1)$ and $B_1=(\alpha',r)$ defined by
\begin{itemize}
\item[(T1)]
$A_1=A_0,\ B_1 = (2m+r+\Delta_A+\Delta_B, r)$,
if either $0<\Delta_A\leq \Delta_B\leq 1$, or
$0<\Delta_A\leq 1,\ 1<\Delta_B\leq 2,\ \Delta_A+\Delta_B<2$;
\item[(T2)]
$A_1 = (2k+r-1-\Delta_A-\Delta_B, r-1),\ B_1 = B_0$,
if either $1\leq \Delta_A\leq \Delta_B\leq 2$, or
$0<\Delta_A\leq 1,\ 1<\Delta_B\leq 2,\ \Delta_A+\Delta_B\geq 2$,
\end{itemize}
are such that $\alpha'-\alpha-1 = t_0 > 2m-2k$, and either $A_1$ or $B_1$ is a light vertex.
Moreover, from a direct inspection we can check that
the thickness of the
dark zone from $A_1$ to $B_1$ is
not greater that the one from $A$ to $B$,
so that by Remark~\ref{r:thfix} we conclude that
\begin{equation}\label{f:thicksnell}
\lng(S(A_1, B_1)) \leq
\lng(S(A, B))\,.
\end{equation}

Let us consider the case (T1),
so that $A_1=A_0$,
and let $B' = (2k+2\kc(\beta)+r, r)$.
The assumption $m-k\geq \kc(\beta)$ implies that
$B'$ lies on the left of $B_0$
(possibly the two points coincide).
Let us
consider the following new paths:
$\Gamma' = S(A_1, B') \cup \cray{B', B_1}$,
and $\Gamma_1 = S(A_1, B_1)$.
Since $2\kc(\beta) < t_0$,
from Corollary~\ref{c:minb}(vi)
we deduce that
\[
\lng(\Gamma') = \ls(2\kc(\beta),\beta) + t_0 +\sqd <
\ls(t_0, \beta) + t_0 + \sqd = \lng(\Gamma_1)\,.
\]
Finally, setting $\widetilde{\Gamma} = \cray{A, A_1} \cup S(A_1, B') \cup \cray{B', B}$ and noticing that
the segments $\cray{A, A_1}$ and $\cray{B, B_1}$ have the same length,
we have that
\[
\lng(\widetilde{\Gamma})=\lng(\Gamma')<
\lng(\Gamma_1)\leq \lng(\Gamma_r),
\]
where the last inequality follows from
(\ref{f:thicksnell}).

The analysis of the case (T2) can be carried out in a similar way,
with obvious modifications.
\end{proof}

%

\section{Geodesics of the chessboard structure ($\beta\geq\sqrt{3/2}$)}
\label{s:geovere}

In this section we shall restrict our analysis to the case $\beta\geq \sqrt{3/2}$, and
we shall provide a complete description of the geodesics joining two light vertices in the chessboard
structure.
The case $\beta < \sqrt{3/2}$ seems to be harder to characterize, as we
shall show by an example (see Example~\ref{r:contro} below).

The next theorem states that, for $\beta > \bc_0$, any geodesic $\Gamma$ joining the origin with the point
$(2n+j,j)$ is a finite union of segments, connecting light vertices,
and lying on light diagonals or on
horizontal lines.

\begin{theorem}\label{t:ottag}
Let $\Gamma$ be a geodesic as in (\ref{f:gamma}).
If $\beta > \bc_0$, then the points $A_k$, $B_k$ are light vertices
for every $k=0,\ldots,n$, and
$\cray{B_{k-1}, A_k}$ is an horizontal segment for every $k=1,\ldots,n$.
\end{theorem}

\begin{proof}
By Corollary \ref{c:minb}(i) we have that
$k_c(\beta)=0$, so that, by Proposition~\ref{p:diag},
$\Gamma$ never cuts a light diagonal.
Hence the points
$A_k$, $B_k$ defined in Proposition \ref{p:akbk}(i) are vertices of light squares, for every $k=0,\ldots,n$.
Given $r=1,\ldots,n$, let $B_{r-1}$ be the exit point from $D_{r-1}$ and $A_r$ be the access point
to $D_r$, and let $i$, $i'=1,\ldots N$ be such that $B_{r-1}\in \coray{P_{i-1}, P_i}$, and
$A_r\in \ocray{P_{i'-1}, P_{i'}}$.
We have that $\theta_i=\theta_{i'}=0$, that
is both $\cray{P_{i-1}, P_i}$ and $\cray{P_{i'-1}, P_{i'}}$ lie on the horizontal sides of the squares.
Indeed, by Proposition \ref{p:akbk}(iii), $\theta_{i'}, \theta_i\in [0,\pi/4)$, and,
if either  $\theta_i$ or $\theta_{i'}$ belong to $(0,\pi/4)$ , then $\Gamma$ should
contain a Snell path starting from a vertex of a light square and lying in a horizontal strip,
a contradiction with Corollary (\ref{c:minb})(ii) and the local
optimality of $\Gamma$.

Then $\theta_{i}=\theta_{i'}=0$, and, by Remark \ref{r:seguno}, there exists two points $A$, $B$ such that the segments
$\cray{B_{r-1},B}$ and $\cray{A,A_r}$ are horizontal, with length greater than
or equal to 1, and they are contained in $\Gamma$.
In order to complete the proof we have to show that $\cray{B,A}$ is a horizontal segment.
Assume by contradiction that this is not the case, that is $B \neq A'$, where $A'$
is the intersection of the line containing $\cray{A,A_r}$ with $D_{r-1}$.
In this case the length of the
polygonal line $\cray{B_{r-1},A',A_r}$ is less than
the length of any curve joining $B_{r-1}$ with $A_r$ and containing $\cray{B_{r-1},B}$ and $\cray{A,A_r}$,
in contradiction with the local minimality of $\Gamma$.
\end{proof}

\begin{dhef}\label{d:stre}
An $S_3$--path is a Snell path joining the point $(2m+k,k)$ to the point $(2m+k+3,k+1)$ for some
$m$, $k\in\N$. We shall denote its normalized length by
\[
\lambda_3=\ls(2,\beta)=
\frac{2}{\sqrt{1-\sn_3^2}}
+\frac{\beta^2}{\sqrt{\beta^2-\sn_3^2}}-2-\sqd,
\]
where $\sn_3 = \hsn(2,\beta)$ is implicitly defined by
\[
\frac{2\sn_3}{\sqrt{1-\sn_3^2}}+\frac{\sn_3}{\sqrt{\beta^2-\sn_3^2}}-1=0.
\]
The optical length of an $S_3$--path will be denoted by $\Lambda_3=\lambda_3+2+\sqd$.
\end{dhef}

\begin{rk}\label{r:r4}
It is straightforward to check that
$\sqrt{10} < \Lambda_3 < \beta\sqrt{10}$,
and, by the very definition of $\bc_0$,
$\Lambda_3 = 2+\sqrt{2}$ when $\beta=\bc_0$.
Moreover, $1/\sqrt{10} < \sn_3 < \beta/\sqrt{10}$.
\end{rk}

The $S_3$--paths will play a fundamental r\^ole in the analysis of the geodesics for
$\sqrt{3/2} \leq \beta < \bc_0$ (see Theorem \ref{t:nottag} below).
Namely for $\beta$ in this range the $S_3$--paths have the minimal normalized length
among all the Snell paths starting from a light vertex $(2m+k,k)$ and reaching a point on the
line $y=k+1$ (see Corollary \ref{c:minb}).

\begin{theorem}\label{t:nottag}
Let $\Gamma$ be a geodesic as in (\ref{f:gamma}).
If $\sqrt{3/2} \leq \beta < \bc_0$, then the points $A_k$, $B_k$ are light vertices
for every $k=0,\ldots,n$.
Moreover $B_{k-1}$ is connected to $A_{k}$ either by an horizontal segment or
by an $S_3$--path for every $k=1,\ldots,n$.
\end{theorem}

\begin{proof}
By Remark \ref{r:esceino} we can assume, without loss of generality, that
$A_0=B_0 = (0,0)$.
Let us define
\begin{equation}\label{f:defi}
i = \min\{r\in \{1,\ldots,n\};\ A_r\ \text{is a light vertex}\}\,.
\end{equation}
\smallskip
Notice that the index $i$ in (\ref{f:defi}) is well defined, since
as a consequence of Proposition \ref{p:akbk}(i) and (ii) at least $A_n$ is a light vertex.
Moreover,
if $i\geq 2$, then $\Gamma$ cuts every light diagonal
$D_k$, $k=1,\ldots, i-1$, that is,
the points $A_k$ and $B_k$ coincide and they are not
light vertices.

Let us denote by $\Gamma'$ the portion of $\Gamma$
joining $A_0=B_0=(0,0)$ to $A_i$.
We are going to prove that
\begin{equation}\label{f:toprove}
\begin{split}
&i = 1,\ \text{and}\\
&\text{$\Gamma'$ is either an horizontal segment
or an $S_3$--path}.
\end{split}
\end{equation}
Once these properties are proved, then $B_1$ is a light vertex
and, repeating the procedure $n$ times, we reach the conclusion.

Let $m\in\N$ be such  that $A_i=(2i+m,m)$.
If $m=0$, then $B_0$ is connected to $A_i$ by an horizontal
segment, so that $i=1$ and (\ref{f:toprove}) holds.
If $m=1$, then by Corollary~\ref{c:minb}(ii) and the local minimality of $\Gamma$,
we conclude that $i=1$ and $\Gamma'$ is and $S_3$--path.
It remains to prove that the case $m\geq 2$ cannot happen.

For every $k = 0,\ldots,i$ let
$r_k=1,\ldots,m$ be such that
$A_k=(2k+\eta_k, \eta_k)\in\overline{\Gamma}_{r_k}$,
where $\Gamma_{r}$ is the Snell path defined in (\ref{f:gammar}).
Let us denote by $C_k=(x_k, r_k-1)$ and $C_k'=(x'_k, r_k)$ respectively the
starting and the ending point of $\Gamma_{r_k}$,
let $\delta_k = x'_k-x_k$ be the thickness of $\Gamma_{r_k}$,
and define
\[
t^-_k = 2k+r_k-1-x_k,\quad
t^+_k = x_k'-2k-r_k,\quad
h_k = \eta_k-\floor{\eta_k}\,,
\]
so that $\delta_k = t^-_k+t^+_k+1$
(see Figure~\ref{fig4}).

\begin{figure}
\includegraphics[width=8cm]{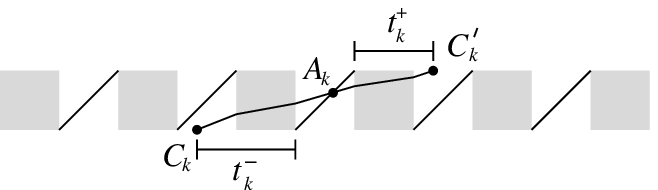}
\caption{}
\label{fig4}
\end{figure}

For $k=0$, we have that $r_0=1$, $C_0=B_0=(0,0)$ and $C_0' = (\delta_0, 1)$.
By construction we have $\delta_0 > 1$.
We claim that $2<\delta_0<3$.
Namely, by Corollary \ref{c:minb}(v) and (iii), we have
\[
\begin{cases}
\lng(\Gamma_1) =\ls(\delta_0,\beta)+\sqrt{2}+\delta_0-1 >\sqrt{2}+\delta_0-1
= \lng(\cray{O,(1,1),(\delta_0,1)})\,, & \text{if}\ \delta_0\in (1,2],\\
\lng(\Gamma_1) >\ls(2,\beta)+\sqrt{2}+\delta_0-1 \geq
\lng\left(\snell(0,(3,1))\cup\cray{(3,1),(\delta_0,1)}\right)\,, & \text{if}\ \delta_0>3\,,
\end{cases}
\]
hence, by the local minimality of $\Gamma_1$, we cannot have neither $1<\delta_0\leq 2$ nor $\delta_0>3$.
For $k=i$ the same arguments show that $2<\delta_i<3$.
Hence, from Proposition \ref{p:diag}, we have that
$1=r_0<r_1<\cdots<r_i=m$.

Finally,
by Lemma~\ref{l:elleh},
\[
\lh(t^-_k, \beta, h_k) > \lh(t^-_k, \beta, 1),\quad
\lh(t^+_k, \beta, 1-h_k) > \lh(t^+_k,\beta, 1), \qquad  k=1,\ldots,i-1,\ (i\geq 2)
\]
where $\lh$ is the function defined in (\ref{f:elleh}).
Moreover, being $\Gamma$ a geodesic, we have $\lh(t^-_k, \beta, h_k)<0$ and
$\lh(t^+_k, \beta, 1-h_k)<0$, so that by Corollary \ref{c:minb}(v), we have $t^-_k, t^+_k>1$.
On the other hand, from Corollary~\ref{c:minb} and Proposition \ref{p:diag},
we cannot have $t^-_k\geq 2$ or $t^+_k\geq 2$,
otherwise $\Gamma_{r_k}$ would intersect more than
one light diagonal.
In conclusion, we have that $t^-_k,t^+_k\in (1,2)$, so that
$\delta_k \in (3, 5)$, $k=1,\ldots,i-1$ ($i\geq 2$).

Let us define
\[
\Delta = \sum_{k=0}^i \delta_k\,.
\]
By construction and from the estimates above we have
\[
 \Delta \leq 2i+m,\quad m\geq i+1,\quad 3i+1<\Delta<5i+1.
\]
It is straightforward to show that
\[
\lng\left(\Gamma'\setminus \bigcup_{k=0}^i\Gamma_{r_k}\right)\geq
\lng(\cray{(\Delta,i+1),(2i+m,m)}),
\]
so that
\[
\lng(\Gamma')\geq
\sum_{k=0}^i \lng(\Gamma_{r_k})+
\sqrt{(2i+m-\Delta)^2+(m-i-1)^2}.
\]
$\lng(\Gamma_{r_k})$ can be estimated using the function
$\lh$.
For $k=0$ and $k=i$ we have that
\begin{equation}\label{f:lenz}
\begin{split}
\lng(\Gamma_{r_0}) = \lh(\delta_0-1,\beta,1)+\delta_0-1+\sqd,\\
\lng(\Gamma_{r_i}) = \lh(\delta_i-1,\beta,1)+\delta_i-1+\sqd\,,
\end{split}
\end{equation}
while, for $k=1,\ldots,i-1$, $(i\geq 2)$
\begin{equation}\label{f:leni}
\lng(\Gamma_{r_k}) =
\lh(t^-_k,\beta,h_k)+\lh(t^+_k,\beta,1-h_k)+\delta_k-1+\sqd\,.
\end{equation}
{}From  Lemma~\ref{r:spiegaz1}(iii) and Lem\-ma~\ref{l:elleh} we have that
\[
\lh(t,\beta,h) >
\lh_h(t,\beta,1)(h-1)+\lh_t(2,\beta,1)(t-2)+\lambda_3
\]
for every $(t,h)\in (1,2)\times [0,1]$.
(We recall that $\lambda_3=\lh(2,\beta,1)$, see Definition~\ref{d:stre}).
Moreover, from (\ref{f:ellet}) and (\ref{f:delleh}) we have that
\[
\lh_t(2,\beta,1) = \sqrt{1-\sn_3^2}-1,\quad
\lh_h(t,\beta,1) = \sqrt{1-\tsn^2(t,\beta,1)}+\tsn(t,\beta,1)-\sqd\,.
\]
where $\sn_3 = \tsn(2,\beta,1)$, see Definition~\ref{d:stre}.
Since $t\mapsto \tsn(t,\beta,1)$ is a decreasing function for $t\geq 0$, and the map
$s\mapsto \sqrt{1-s^2}+s-\sqrt{2}$ is increasing for $0\leq s\leq 1/\sqrt{2}$, we finally obtain
\begin{equation}\label{f:convex}
\lh(t,\beta,h) > \lambda_3+\left(\sqrt{1-\sn_3^2}-1\right)(t-2)+
\left(\sqrt{1-\sn_3^2}+\sn_3-\sqd\right)(h-1)
\end{equation}
for every $(t,h)\in (1,2)\times[0,1]$.
{}From (\ref{f:lenz}), (\ref{f:leni}) and (\ref{f:convex}) we obtain
\begin{equation}\label{f:elenz}
\begin{split}
\lng(\Gamma_{r_0}) & > \lambda_3+\left(\sqrt{1-\sn_3^2}-1\right)(\delta_0-3)+\delta_0-1+\sqd,\\
\lng(\Gamma_{r_i}) & > \lambda_3+\left(\sqrt{1-\sn_3^2}-1\right)(\delta_i-3)+\delta_i-1+\sqd,\\
\lng(\Gamma_{r_k}) & > 2\lambda_3+\left(\sqrt{1-\sn_3^2}-1\right)(\delta_k-5)
-\sqrt{1-\sn_3^2}-\sn_3+\delta_k-1+2\sqd,
\end{split}
\end{equation}
($k=1,\ldots,i-1$, $i\geq 2$), so that
\[
\sum_{k=0}^i \lng(\Gamma_{r_k}) >
\Delta\sqrt{1-\sn_3^2} +\sn_3+\left(2\sqd+4-6\sqrt{1-\sn_3^2}-\sn_3+2\lambda_3\right)i
\]
and
\begin{equation}\label{f:stimax}
\begin{split}
\lng(\Gamma') > {} & \Delta\sqrt{1-\sn_3^2} +\sn_3+\left(2\sqd+4-6\sqrt{1-\sn_3^2}-\sn_3+2\lambda_3\right)i\\
& +\sqrt{(2i+m-\Delta)^2+(m-i-1)^2}\,.
\end{split}
\end{equation}
Let us now consider the path $\Gamma''$ starting from the origin,
obtained by the concatenation of
$i$ $S_3$-paths and the segment connecting the point
$(3i,i)$ to $A_i = (2i+m,m)$.
Since this segment connects two points on the light diagonal $D_i$, its
length is $(m-i)\sqd$, so that
\[
\lng(\Gamma'') = i\Lambda_3+(m-i)\sqd =
i(\lambda_3+2+\sqd)+(m-i)\sqd.
\]
We are going to show that $\lng(\Gamma'') < \lng(\Gamma')$,
in contradiction with the local minimality of $\Gamma'$.
We have that
\begin{equation}\label{f:stimaxx}
\begin{split}
\lng(\Gamma')-\lng(\Gamma'') > {} &
\left(2+\sqd-3\sqrt{1-\sn_3^2}-\sn_3+\lambda_3\right)i
-\left(\sqd-\sqrt{1-\sn_3^2}-\sn_3\right)\\
& +\sqrt{\mu_1^2+\mu_2^2}-\mu_1 \sqrt{1-\sn_3^2}
-\mu_2\left(\sqd-\sqrt{1-\sn_3^2}\right)
\end{split}
\end{equation}
where
\[
\mu_1 = 2i+m-\Delta\geq 0,\quad \mu_2 = m-i-1\geq 0.
\]
Since $3i+1<\Delta$ and $m\geq i+1$ we have that
\begin{equation}\label{f:stimaxxaa}
0\leq\mu_1<\mu_2.
\end{equation}
Moreover
\begin{equation}\label{f:stimaxxa}
\sqrt{\mu_1^2+\mu_2^2}-\mu_1 \sqrt{1-\sn_3^2}
-\mu_2\left(\sqd-\sqrt{1-\sn_3^2}\right) =
\mu_2\, \varphi\left(\frac{\mu_1}{\mu_2}\right)\,,
\end{equation}
where
\[
\varphi(s) = \sqrt{1+s^2}-s\,\sqrt{1-\sn_3^2}-\left(\sqd-\sqrt{1-\sn_3^2}\right)\,.
\]
Since $0<\sn_3<1/\sqd$, we have that $1/\sqd < \sqrt{1-\sn_3^2} < 1$.
It can be easily checked that $\varphi'(s)<0$ for every $s\in [0,1]$, hence
\begin{equation}\label{f:stimaxxb}
0=\varphi(1) < \varphi(s),\qquad
\forall s\in [0,1).
\end{equation}
{}From (\ref{f:stimaxx}), (\ref{f:stimaxxaa}), (\ref{f:stimaxxa})
and (\ref{f:stimaxxb}) we thus get
\begin{equation}\label{f:stimaxy}
\begin{split}
\lng(\Gamma')-\lng(\Gamma'') > {} &
\left(2+\sqd-3\sqrt{1-\sn_3^2}-\sn_3+\lambda_3\right)i
-\left(\sqd-\sqrt{1-\sn_3^2}-\sn_3\right)\,.
\end{split}
\end{equation}
We claim that
\begin{equation}\label{f:stimaxxx}
\begin{split}
2+\sqd-3\sqrt{1-\sn_3^2}-\sn_3+\lambda_3
>\sqd-\sqrt{1-\sn_3^2}-\sn_3> 0\,.
\end{split}
\end{equation}
The second inequality
in (\ref{f:stimaxxx}) easily follows from the fact that $0<\sn_3<1/\sqd$.
Concerning the first one,
by the very definition of $\lambda_3$,
and since $b\mapsto b^2/\sqrt{b^2-\sn_3^2}$ is an
increasing function in $[1,+\infty)$, we have that
\[
\begin{split}
& \left(2+\sqd-3\sqrt{1-\sn_3^2}-\sn_3+\lambda_3\right)
-\left(\sqd-\sqrt{1-\sn_3^2}-\sn_3\right)\\
& =
-2\sqrt{1-\sn_3^2}+\frac{2}{\sqrt{1-\sn_3^2}}
+\frac{\beta^2}{\sqrt{\beta^2-\sn_3^2}}-\sqd\\
& \geq
-2\sqrt{1-\sn_3^2}+\frac{2}{\sqrt{1-\sn_3^2}}
+\frac{3/2}{\sqrt{3/2-\sn_3^2}}-\sqd =: \psi(\sn_3).
\end{split}
\]
It can be checked that $\psi$ is strictly increasing in $(0,1)$.
Since, by Remark \ref{r:r4}, $\sn_3>1/\sqrt{10}$, we have that
\[
\psi(\sn_3) > \psi\left(\frac{1}{\sqrt{10}}\right) =
\frac{\sqrt{10}}{15}+\frac{3\sqrt{10}}{2\sqrt{14}}-\sqrt{2}>0\,,
\]
and (\ref{f:stimaxxx}) is proved.

Since $i\geq 1$, it is straightforward to check that
(\ref{f:stimaxy}) and (\ref{f:stimaxxx})
imply that $\lng(\Gamma')-\lng(\Gamma'')>0$, in contradiction with the local minimality
of $\Gamma'$.
\end{proof}

\begin{theorem}\label{t:lenb}
Let $\Gamma$ be a geodesic from the origin to
a light vertex $\xi = (x,y)$, with $0\leq y\leq x$.
\begin{itemize}
\item[(i)]
If $\beta\geq\bc_0$, then
$\lng(\Gamma) = x + (\sqd-1)y$.

\item[(ii)]
If $\sqrt{3/2} \leq \beta \leq \bc_0$, then
\[
\lng(\Gamma) =
\begin{cases}
x+(\Lambda_3-3)y,
&\text{if\ $0\leq y\leq x/3$},\\
\dfrac{\Lambda_3-\sqd}{2}\, x + \dfrac{3\sqd - \Lambda_3}{2}\, y,
&\text{if\ $x/3\leq y\leq x$}\,,
\end{cases}
\]
where $\Lambda_3 = \Lambda_3(\beta)$ is the length of an
$S_3$--path, introduced in Definition~\ref{d:stre}.
\end{itemize}
\end{theorem}

\begin{proof}
(i) For $\beta > \bc_0$
it is a straightforward consequence of Theorem~\ref{t:ottag}.

(ii) Let us consider the case $\sqrt{3/2} \leq \beta < \bc_0$.
From Theorem~\ref{t:nottag} we know that $\Gamma$ is the
concatenation of $S_3$--paths
and segments joining light vertices,
lying on light diagonals or on horizontal lines.
Hence we have that
\[
\lng(\Gamma) = t\, \Lambda_3 + r + d\sqd,
\]
where $t$ is the number of $S_3$--paths,
$r$ is the number of unit horizontal segments,
and $d$ is the number of diagonals of light squares.
It is clear that the three numbers $t,r,d\in\N$ must satisfy
the constraints
\[
3t+r+d = x,\quad
t+d=y\,,
\]
so that
\[
\lng(\Gamma)=(\Lambda_3-2-\sqd) t + x + (\sqd-1) y\,.
\]
Since $\Lambda_3 < 2+\sqd$, $L$ is minimized by choosing the largest
admissible value of $t$,
which is
$y$ if $y\leq x/3$,
and $(x-y)/2$ if $y\geq x/3$.
(We remark that $(x-y)/2$ is an integer number,
since the point $P=(x,y)$ is a light vertex.)
In conclusion, if $y\leq x/3$ we choose $t=y$, $d=0$ and $r=x-3y$,
whereas if $y\geq x/3$ we choose
$t = (x-y)/2$, $d = (3y-x)/2$, $r=0$,
obtaining (ii).

Finally, if $\beta=\bc_0$, it is straightforward to check that
the lengths of geodesics can be computed indifferently as in (i) or in (ii).
We remark that, in this case, these two formulas give the same result,
since $\Lambda_3 = 2+\sqd$.
\end{proof}

One may wonder if the previous characterization of the geodesics
for $\sqrt{3/2}\leq\beta < \bc_0$ remains valid for
$\bc_1<\beta<\bc_0$.
The following example shows that this is not the case.

\begin{ehse}\label{r:contro}
Let $\Gamma = \cray{(0,0),(1,1)}\cup\snell((1,1),(4,2))$.
By Theorem~\ref{t:lenb},
if $\sqrt{3/2}\leq\beta < \bc_0$, then $\Gamma$ is a geodesic joining
the origin to the point $\xi = (4, 2)$.
Given $t\in [0,1]$, let us consider the curve
\[
\Gamma(t) = \snell((0,0), (1+t,1))
\cup \snell((1+t, 1), (4,2)).
\]
We have that
\[
L(t) := \lng(\Gamma(t)) =
\ls(t,\beta) + \ls(2-t,\beta)\,.
\]
Moreover, for $t=0$ we have
$\Gamma(0) = \Gamma$, $L(0) = \Lambda_3+\sqd$, and
\begin{equation}\label{f:lpri}
\begin{split}
L'(0) & = \lim_{t\to 0^+}\frac{L(t)-L(0)}{t} =
\ls_t^+(0,\beta)-\ls_t^-(2,\beta)
\\ & =
\sqrt{\beta^2-\frac{1}{2}}-\sqrt{1-\sn_3^2}\,.
\end{split}
\end{equation}
We recall that, for a given $\beta>1$,
$\sn_3$ is the unique zero in $(0,1)$ of the function
\[
g(\sn) = \frac{2\sn}{\sqrt{1-\sn^2}}+\frac{\sn}{\sqrt{\beta^2-\sn^2}}-1\,.
\]
Moreover, $g$ is a strictly monotone increasing function in $(0,1)$,
with $g(0) = -1$ and $g(s) \to +\infty$ as $s\to 1^-$.
For $\beta\leq\sqrt{3/2}$, let us compute
\[
\psi(\beta) = g\left(\sqrt{\frac{3}{2}-\beta^2}\right)=
2\sqrt{\frac{3-2\beta^2}{2\beta^2-1}}
+\sqrt{\frac{3-2\beta^2}{4\beta^2-3}}-1\,.
\]
Since
\[
\psi'(\beta) = -\frac{2\beta}{\sqrt{3-2\beta^2}}\left(
\frac{3}{(4\beta^2-3)^{3/2}}+\frac{4}{(2\beta^2-1)^{3/2}}\right)<0,
\]
the map $\psi$ is strictly monotone decreasing in $[1, \sqrt{3/2}]$,
with $\psi(1) = 1$ and $\psi\left(\sqrt{3/2}\right)=-1$.
The unique zero of $\psi$ in $(1, \sqrt{3/2})$ is
$\tilde{\beta}\simeq 1.17868$, and $\tilde{\beta} > \bc_1 \simeq 1.06413$.
Hence, for $1<\beta<\tilde{\beta}$, we have that
$\sn_3 < \sqrt{3/2-\beta^2}$ and,
by (\ref{f:lpri}), $L'(0)<0$.
In conclusion, if $1<\beta<\tilde{\beta}$,
for $t>0$ small enough
we have that $L(t) < L(0)$,
and $\Gamma$ is not a geodesic.
\end{ehse}

\section{The homogenized metric}\label{homet}

As a direct consequence of Theorem \ref{t:lenb}, we obtain the complete description of the
homogenized metric $\varPhi_\beta$ for $\beta\geq \sqrt{3/2}$.
In the general case we discuss the regularity of the homogenized metric.

In order to make some usefull reductions, we need two remarks on
the distance ${d_{\beta}^{\varepsilon}} (0,\xi)$ defined in (\ref{f:defdist}).

\begin{rk}
Since $|\xi-\eta|\leq {d_{\beta}^{\varepsilon}} (\eta,\xi)\leq \beta |\xi-\eta|$,
it can be easily seen that
\begin{equation}\label{f:dapprx}
{\varPhi_\beta}(\xi)=\lim_{\varepsilon \to 0^+} {d_{\beta}^{\varepsilon}}(\eta_\varepsilon,\xi_\varepsilon)\,,
\qquad \forall\ \xi\in\R^2,\
\forall \xi_\varepsilon \to \xi\,,\
\forall \eta_\varepsilon \to 0\,.
\end{equation}
\end{rk}


\begin{rk}\label{r:symmetry}
For every $\varepsilon>0$ let $\eta_\varepsilon = (\frac{\varepsilon}{2}\,,\frac{\varepsilon}{2})$.
{}From (\ref{f:dapprx}) we have that
\[
{\varPhi_\beta}(\xi)=\lim_{\varepsilon \to 0^+}
d_{\beta}^{\varepsilon}(\eta_\varepsilon, \xi+\eta_\varepsilon)\,,
\qquad \forall \xi\in\R^2\,.
\]
Since the map $\xi\mapsto d_{\beta}^{\varepsilon}(\eta_\varepsilon, \xi+\eta_\varepsilon)$
is symmetric w.r.t.\ the coordinated axis and the diagonals passing through the origin,
it is clear that $\varPhi_\beta$ has the same symmetries.
\end{rk}

In what follows, given $(x,y)\in\R^2$ we denote
\[
M = \max\{|x|, |y|\},\qquad
m = \min\{|x|, |y|\}.
\]

\begin{theorem}\label{t:Phi}
For every $(x,y)\in\R^2$ the following hold.
\begin{itemize}
\item[(i)]
If $\beta\geq\bc_0$, then
${\varPhi_\beta}(x,y) = M + (\sqd-1)m$.
\item[(ii)]
If $\sqrt{3/2} \leq \beta \leq \bc_0$, then
\[
\varPhi_\beta(x,y) =
\begin{cases}
M+(\Lambda_3-3)m,
&\text{if $3m\leq M$},\\
\dfrac{\Lambda_3-\sqd}{2}\, M + \dfrac{3\sqd - \Lambda_3}{2}\, m,
&\text{if $3m\geq M$}\,,
\end{cases}
\]
where $\Lambda_3 = \Lambda_3(\beta)$ is the length of an
$S_3$--path, introduced in Definition~\ref{d:stre}.
\end{itemize}
\end{theorem}

\begin{proof}
By Remark~\ref{r:symmetry} it is enough to consider the
case $0\leq y\leq x$, so that $M=x$ and $m=y$.
Let $\xi = (x,y)$ and,
for $\varepsilon>0$, let us define
\[
j = \floor{\frac{y}{\varepsilon}},\quad
n = \floor{\frac{x-j\varepsilon}{2\varepsilon}},\qquad
\xi_{\varepsilon} =(x_{\varepsilon}, y_{\varepsilon}) = ((2n+j)\varepsilon, j\varepsilon).
\]
Then
\[
|y-y_{\varepsilon}|<\varepsilon,\quad
|x-x_{\varepsilon}|<2\varepsilon,
\]
so that $|\xi-\xi_{\varepsilon}| < \varepsilon\sqrt{5}$.
Moreover
$d_\beta^{\varepsilon}(0,\xi_{\varepsilon})$
is explicitly computed in Theorem~\ref{t:lenb}.
Since $\xi_{\varepsilon}\rightarrow \xi$, by
(\ref{f:dapprx}), the conclusion follows.
\end{proof}

\begin{figure}
\includegraphics[height=4cm]{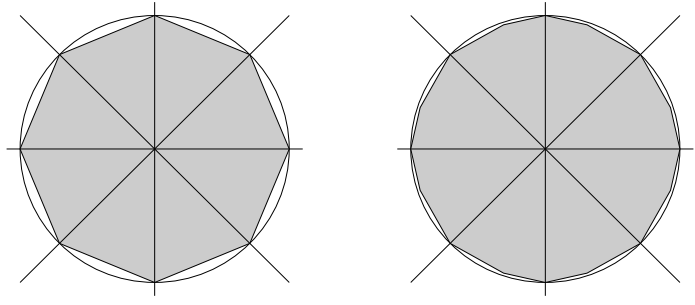}
\caption{The homogenized unit ball for $\beta\geq\bc_0$ (left)
and $\sqrt{3/2}\leq\beta<\bc_0$ (right)}
\label{fig:ottag}
\end{figure}

In the general case we have the following result.

\begin{theorem}\label{t:edges}
Let $\beta>1$ be given, and let $\kc(\beta)$ be the number defined in Definition \ref{d:kc}.
Then
$\varPhi_\beta(x,y)=M+\left(\ls(2\kc(\beta), \beta\right)+\sqrt{2}-1)m$
for every $(x,y)\in\R^2$ belonging to one of the cones
$\{(2\kc(\beta)+1)|y|\leq |x|\}$ or
$\{(2\kc(\beta)+1)|x|\leq |y|\}$.
\end{theorem}

\begin{proof}
{}From Remark~\ref{r:symmetry} it is enough to consider the case
$0\leq (2\kc(\beta)+1)y\leq x$.
Moreover, we can assume that $y>0$, the case $y=0$ being trivial.
Since the homogenized metric depends continuously on $\beta$, it is not restrictive to assume
$\beta\neq\bc_k$, in such a way that
\[
\ls(2\kc(\beta), \beta)<\ls(t, \beta)\,, \qquad \forall\ t\geq 0,\ t\neq 2\kc(\beta)
\]
(see Corollary~\ref{c:minb}(iii)).
For every $0<\varepsilon<y$, let $\xi_{\varepsilon}=(x_\varepsilon,y_\varepsilon)$
be the nearest light vertex to $\xi=(x,y)$
below the line $x=(2\kc(\beta)+1)y$.
Then $\xi_{\varepsilon}=\left((2n+j)\varepsilon, j\varepsilon\right)$ for
some $j\geq 1$ and $n\geq \kc(\beta)j$, and $\xi_\varepsilon \to \xi$ as $\varepsilon$ tends to 0.

We claim that
\begin{equation}\label{f:edgeps}
{d_{\beta}^{\varepsilon}}(0, \xi_\varepsilon)= x_\varepsilon+
(\ls(2\kc(\beta), \beta)+\sqrt{2}-1)y_\varepsilon\,, \qquad
\forall\ 0<\varepsilon<y\,,
\end{equation}
so that the result will follows from (\ref{f:dapprx}).
In order to prove the claim,
after a scaling, we have to depict a geodesic $\Gamma$ joining the origin to the
point $(2n+j, j)$ in the standard chessboard structure.
Let us define the class $\mathcal{S}$ 
of all Snell paths
joining the light vertices $(2m+r-1,r-1)$ and $(2m+2\kc(\beta)+r,r)$, $m,\ r\in\Z$.
We are going to show that  $\Gamma$ has to be the concatenation of
$j$ Snell paths in $\mathcal{S}$ and of horizontal segments.
As a consequence, since the length of any path in $\mathcal{S}$
equals to $\ls(2\kc(\beta),\beta)+2\kc(\beta)+\sqrt{2}$,
whereas the total length of the horizontal segments is $2(n-\kc(\beta)j)$,
we obtain that (\ref{f:edgeps}) holds.

Let us define the paths $\Gamma_1,\ldots,\Gamma_j$ as in (\ref{f:gammar}).
For every $r=1,\ldots,j$ let us denote by
$(x_r, r-1)$ and $(x_r', r)$ the endpoints of
$\overline{\Gamma_r}$, and let $\tau_r = x_r'-x_r$.
Thanks to Lemma \ref{l:disnell}, we have that
\[
\lng(\Gamma_r)\geq \ls(2\kc(\beta),\beta)+\tau_r-1+\sqrt{2}\,,
\]
for every $r=1,\ldots, j$. Then we get
\[
\begin{split}
\lng(\Gamma)&=\sum_{r=1}^j \lng(\Gamma_r)+2n+j-\sum_{r=1}^j\tau_r
\geq 2n+j+(\ls(2\kc(\beta),\beta)+\sqrt{2}-1)j
\end{split}
\]
and again by Lemma \ref{l:disnell} the equality holds if and only if $\tau_r=2\kc(\beta)$  and
$\Gamma_r\in \mathcal{S}$
for every $r=1,\ldots,j$.
\end{proof}

\begin{cor}\label{c:regul}
For every $\beta>1$ the unit ball of the homogenized metric is not strictly convex, and
its boundary is not differentiable.
\end{cor}

\begin{rk}
The presence of faces in the optical ball corresponds
to nonuniqueness of the geodesics. More precisely, if $F$ is a face
of positive length, and $C=\{\lambda\eta ;\ \eta \in F,\ \lambda\geq 0\}$
is the corresponding cone, then for every $\xi\in C$, a function
$u\in AC([0,1],\R^2)$ with
$u(0) = 0$, $u(1) = \xi$,
parameterizes a geodesic if and only if
$u'(t)\in C$ for a.e.~$t\in [0,1]$.

Namely, there exists $p\in\R^2$ such that
$\Phi_{\beta}(\eta) = \pscal{p}{\eta}$ for every $\eta\in C$, and
$\Phi_{\beta}(\eta) > \pscal{p}{\eta}$
for every $\eta\in \R^2\setminus C$.
Hence, if $u'(t)\in C$ for a.e.~$t\in [0,1]$, then we get
\[
\mathcal{L}_\beta^{hom}(u)=\int_0^1 {\varPhi_\beta}(u'(t))\, dt =
\int_0^1 \pscal{p}{u'(t)}\, dt = \Phi_{\beta}(\xi),
\]
whereas $\mathcal{L}_\beta^{hom}(u)>\Phi_{\beta}(\xi)$
whenever
$\{t\in [0,1];\ u'(t)\not\in C\}$
has positive measure.
\end{rk}

As a final remark,
let us consider the chessboard structure corresponding to the
upper semicontinuous function
\begin{equation}\label{f:defat}
\tilde{a}_\beta(\xi)=
\begin{cases}
\beta & \textrm{if\ } \xi\in \left([0,1] \times [1,2]\right)
\cup \left([1,2] \times [0,1]\right) \\
1 & \textrm{otherwise}
\end{cases}
\end{equation}
which differs from the standard chessboard structure defined
in (\ref{f:defa}) by the fact that $\tilde{a} = \beta$
instead of $1$
on the sides of the squares.

In this way we obtain a new family of length functionals
$\widetilde{\mathcal{L}}_\beta^{\varepsilon}$.
In this case the existence of a geodesic joining the origin
with a point $\xi\in\R^2$ is not guaranteed.
For example, if $\xi = (\varepsilon,0)$, we have
$\widetilde{\mathcal{L}}_\beta^{\varepsilon}(u) > \varepsilon$
for every $u\in AC([0,1], \R^2)$ such that
$u(0) = 0$ and $u(1) = \xi$.
On the other hand, we can construct a minimizing sequence
$(u_n)_n$ such that $\widetilde{\mathcal{L}}_\beta^{\varepsilon}(u_n) \to \varepsilon$
for $n\to +\infty$,
defining
\[
u_n(t) =
\begin{cases}
(\varepsilon t, 2t/n),
&\text{if $t\in [0,1/2]$},\\
(\varepsilon t, 2(1-t)/n),
&\text{if $t\in [1/2,1]$}\,.
\end{cases}
\]
Nevertheless, the $\Gamma$-limit with respect to the $L^1$-topology of
the functionals $(\widetilde{\mathcal{L}}_\beta^{\varepsilon})$ coincides with
the $\Gamma$-limit $\mathcal{L}_\beta^{hom}$ of the functionals $(\mathcal{L}_\beta^{\varepsilon})$.
Namely,
the liminf inequality is certainly satisfied since
$\widetilde{\mathcal{L}}_\beta^{\varepsilon}\geq {\mathcal{L}}_\beta^{\varepsilon}$.
On the other hand, given $u\in AC([0,1], \R^2)$ and
a recovering sequence $(u_{\varepsilon})$ for $({\mathcal{L}}_\beta^{\varepsilon})$,
we can construct
a recovering sequence $(\tilde{u}_{\varepsilon})$ for $(\widetilde{\mathcal{L}}_\beta^{\varepsilon})$
in the following way:
for a given $\varepsilon$, we obtain $\tilde{u}_{\varepsilon}$ modifying $u_{\varepsilon}$
in the region where $u_{\varepsilon}(t)$ belongs to the set $S$ of the sides of squares,
in such a way that
${\mathcal{L}}_\beta^{\varepsilon}(\tilde{u}_{\varepsilon})<{\mathcal{L}}_\beta^{\varepsilon}({u}_{\varepsilon})+\varepsilon$,
and the set
$\{t\in[0,1];\ \tilde{u}_{\varepsilon}\in S\}$
has vanishing Lebesgue measure.

%
\bibliographystyle{/Ga/BibTeX/mybst}
\bibliography{/Ga/BibTeX/Graziano,/Ga/BibTeX/Ricerca,/Ga/BibTeX/Physics}
\end{document}